\documentclass[11pt]{amsart}
\usepackage{amsmath,amssymb}
\usepackage[latin1]{inputenc}

\usepackage{amsthm}
\usepackage{stmaryrd}
\usepackage{mathrsfs}
\usepackage[dvips]{graphicx}

\usepackage{color}
\usepackage[normalem]{ulem}
\usepackage{cancel}

\topskip=\baselineskip  \headsep=0.3in
\newtheorem{theorem}{Theorem}[section]
\newtheorem{lemma}[theorem]{Lemma}
\newtheorem{proposition}[theorem]{Proposition}

\newtheorem{definition}[theorem]{Definition}
\newtheorem{corollary}[theorem]{Corollary}

\newcommand{\nsp}{\upharpoonleft \hspace{-4.5pt}\downharpoonleft}

\newcommand{\bo}{{\mathcal{O}}}

\newcommand{\D}{{\mathcal{D}}}
\newcommand{\q}{{\mathcal{Q}}}

\newcommand{\Gal}{{\rm Gal}}

\newcommand{\lcm}{\mbox{\rm lcm}}
\newcommand{\Cen}{\mbox{\rm C}}
\newcommand{\core}{\mbox{\rm Core}}
\newcommand{\Aut}{\mbox{\rm Aut}}

\newcommand{\inv}{^{-1}}

\newcommand{\Z}{{\mathbb Z}}
\newcommand{\Q}{{\mathbb Q}}
\newcommand{\R}{{\mathbb R}}
\newcommand{\C}{{\mathbb C}}
\newcommand{\HQ}{{\mathbb H}}

\renewcommand{\O}{\mathcal{O}}

\newcommand{\GL}{{\rm GL}}
\newcommand{\SL}{{\rm SL}}

\newcommand{\PSL}{{\rm PSL}}

\newcommand{\matriz}[1]{\begin{array} #1 \end{array}}

\newcommand{\GEN}[1]{\langle #1 \rangle}

\newcommand{\quat}[2]{\left( \frac{#1}{#2} \right)}

\newcommand{\A}{{\mathcal A}}
\newcommand{\B}{{\mathcal B}}
\newcommand{\U}{\mathcal{U}}

\newcommand{\diag}{\operatorname{diag}}

\title{On the Congruence Subgroup Problem for integral group rings}

\author{Mauricio Caicedo}
\address{Departamento de Matem\'{a}ticas, Universidad de Murcia,  30100 Murcia, Spain}
\email{mjose.caicedo@um.es}

\author{\'{A}ngel del R\'{\i}o}
\address{Departamento de Matem\'{a}ticas, Universidad de Murcia,  30100 Murcia, Spain}
\email{adelrio@um.es}

\thanks{This research is partially supported by the Spanish Government under Grant MTM2012-35240 with "Fondos FEDER" and Fundaci\'{o}n S\'{e}neca of Murcia under Grant 04555/GERM/06}

\subjclass[2010]{19B37, 20H05; 20C05}

\keywords{Congruence Subgroup Problem; Integral group rings}

\begin{document}

\begin{abstract}
Let $G$ be a finite group, $\Z G$ the integral group ring of $G$ and $\U(\Z G)$ the group of units of $\Z G$. The Congruence Subgroup Problem for $\U(\Z G)$ is the problem of deciding if every subgroup of finite index of $\U(\Z G)$ contains a congruence subgroup, i.e. the kernel of the natural homomorphism $\U(\Z G) \rightarrow \U(\Z G/m\Z G)$ for some positive integer $m$.
The congruence kernel of $\U(\Z G)$ is the kernel of the natural map from the completion of $\U(\Z G)$ with respect to the profinite topology to the completion with respect to the topology defined by the congruence subgroups. The Congruence Subgroup Problem has a positive solution if and only if the congruence kernel is trivial.
We obtain an approximation to the problem of classifying the finite groups for which the congruence kernel of $\U(\Z G)$ is finite. More precisely, we obtain a list $L$ formed by three families of finite groups and 19  additional groups such that if the congruence kernel of $\U(\Z G)$ is infinite then $G$ has an epimorphic image isomorphic to one of the groups of $L$.
About the converse of this statement we at least know that if one of the 19 additional groups in $L$ is isomorphic to an epimorphic image of $G$ then the congruence kernel of $\U(\Z G)$ is infinite.
However, to decide for the finiteness of the congruence kernel in case $G$ has an epimorphic image isomorphic to one of the groups in the three families of $L$ one needs to know if the congruence kernel of the group of units of an order in some specific division algebras is finite and this seems a difficult problem.

\end{abstract}

\maketitle

\section{Introduction}

Let $A$ be a finite dimensional semisimple rational algebra and $R$ a $\Z$-order in $A$.
Let $\U(R)$ denote the group of units of $R$.
For a positive integer $m$  let $\U(R,m)$ denote the kernel of the natural group homomorphism $\U(R)\rightarrow \U(R/mR)$, i.e.
	$$\U(R,m) = \{u\in \U(R) : u-1\in mR\}.$$
More generally, if $n$ and $m$ are positive integers then $M_n(R)$ denotes the $n\times n$ matrix ring with entries in $R$, $\GL_n(R)$ denotes the group of units of $M_n(R)$, $\SL_n(R)$ denotes the subgroup of $\GL_n(R)$ formed by the elements of reduced norm 1, $\GL_n(R,m)=\U(M_n(R),m)$ and $\SL_n(R,m)=\SL_n(R)\cap \GL_n(R,m)$.

A subgroup of $\U(R)$ (respectively, $\SL_n(R)$) containing $\U(R,m)$ (respectively, $\SL_n(R,m)$) for some positive integer $m$ is called a congruence subgroup of $\U(R)$ (respectively, $\SL_n(R)$).
If $m$ is a positive integer then $R/mR$ is finite. Hence $\U(R,m)$ has finite index in $\U(R)$ and $\SL_n(R,m)$ has finite index in $\SL_n(R)$.
Therefore every congruence subgroup of $\U(R)$ (respectively, $\SL_n(R)$) has finite index in $\U(R)$ (respectively, $\SL_n(R)$).
The \emph{Congruence Subgroup Problem} (\emph{CSP}, for brevity) asks for the converse of this statement.
More precisely, we say that the CSP has a positive solution  for $R$ (respectively, for $\SL_n(R)$) if every subgroup of finite index in $\U(R)$ (respectively, in $\SL_n(R)$) is a congruence subgroup.
See \cite{PrasadRapinchuk} for a survey on a much more general version of the Congruence Subgroup Problem.

Serre introduced a quantitative version of the Congruence Subgroup Problem.
Let $\mathcal{F}$ be the set of normal subgroups of finite index in $\U(R)$ and let $\mathcal{C}$ denote the set of congruence subgroups of $\U(R)$.
Both $\mathcal{F}$ and $\mathcal{C}$ define bases of neighborhoods of $1$ for group topologies in $\U(R)$.
The corresponding completions are the projective limits
    $$\U_{\mathcal{F}}(R) = \varprojlim_{N\in \mathcal{F}} \U(R)/N \quad\text{and}\quad \U_{\mathcal{C}}(R) = \varprojlim_{N\in \mathcal{C}} \U(R)/N.$$
The identity map of $\U(R)$ induces a surjective group homomorphism $\U_{\mathcal{F}}(R)\rightarrow \U_{\mathcal{C}}(R)$ and the kernel of this homomorphism is called the \emph{congruence kernel} of $\U(R)$.
The congruence kernel of $\SL_n(R)$ is defined similarly.
It is easy to see that the CSP has a positive solution if and only if the congruence kernel is trivial.
The quantitative version  of the CSP is the problem of calculating the congruence kernel.
If $S$ is another $\Z$-order in $A$ then $mR\subseteq S$ and $mS\subseteq R$ for some positive integer $m$ and $\U(R)\cap \U(S)$ has finite index in both $\U(R)$ and $\U(S)$.
Using this it is easy to see that the congruence kernel of $\U(R)$ is finite if and only if the congruence kernel of $\U(S)$ is finite. In that case we say that $A$ has the CSP property.

%\ChAngel{HE MOVIDO EL SIGUIENTE P\'{A}RRAFO DE SITIO. LO QUE HE BORRADO EN EL P\'{A}RRAFO ANTERIOR SE PUEDE QUITAR PORQUE EST\'{A} EN EL SIGUIENTE P\'{A}RRAFO}

Assume now that $A=\prod_{i=1}^k M_{n_i}(D_i)$ is the Wedderburn decomposition of $A$ (i.e. each $n_i$ is a positive integer and $D_i$ a finite dimensional rational division algebra) and let $R_i$ be an order in $D_i$ for every $i$.
Then $\prod_{i=1}^k M_{n_i}(R_i)$ is an order in $A$, and for each $i$, $Z(R_i)$ is a $\Z$-order in $Z(D_i)\cong Z(M_{n_i}(D_i))$,  $\U(Z(R_i))\cap \SL_n(R_i)$ is finite and $\GEN{\U(Z(R_i)),\SL_n(R_i)}$ has finite index in $\GL_n(R)$.
Combining this with the facts that number fields have the CSP property and that the order chosen to check the CSP property does not make a difference, it is easy to see that $A$ has the CSP property if and only if each non-commutative component $M_{n_i}(D_i)$ has the CSP property if and only if  the congruence kernel of $\SL_{n_i}(R_i)$ is finite for every $i=1,\dots,k$.

The CSP for $\SL_n(R)$, with $R$ an order in a finite dimensional division algebra $D$ has been widely studied and solved except for the case when $n=1$ and $D$ is non-commutative (see e.g. \cite{PrasadRapinchuk}).
We call \emph{exceptional algebras} to the  algebras of one of the following two types:
\begin{itemize}
\item [(EC1)] A non-commutative finite dimensional division rational algebra which is not a totally definite quaternion algebra.
\item [(EC2)] A two-by-two matrix ring over $D$ with $D=\Q$, an imaginary quadratic extension of $\Q$ or a totally definite quaternion algebra over $\Q$.
\end{itemize}
By results of \cite{Bass-Milnor-Serre,Vaserstein,Liehl,Venkataramana1994} every finite dimensional simple algebra not having the CSP property is exceptional. Moreover, the exceptional algebras of type (EC2) have not the CSP property.

This paper addresses the CSP for integral group rings of finite groups.
More precisely our aim is to classify the finite groups $G$ for which  the congruence kernel of  $\U(\Z G)$ is finite (equivalently, the rational group algebra $\Q G$ has the CSP property).
Besides the intrinsic interest of this question its solution has applications in the study of the group of units of $\Z G$, not only because a  solution for the CSP provides relevant information on the normal subgroups of $\U(\Z G)$ but also because it can be used to obtain generators of a subgroup of finite index in $\U(\Z G)$ as it has been shown in \cite{RitterSehgalNilp} and \cite{JespersLealManus}.

We now explain our strategy.
By the  discussion above,  the congruence kernel of $\U(\Z G)$ is finite if and only if every non-commutative simple component of $\Q G$ has the CSP property.
If $N$ is a normal subgroup of $G$ then every simple component of $\Q(G/N)$ is a simple component of $\Q G$.
Therefore, if $\Q G$ has the CSP property then $\Q(G/N)$ has the CSP property.
This suggests the following approach to the problem.
We say that $G$ is \emph{CSP-critical} if $\Q G$ has not the CSP property but $\Q(G/N)$ has the CSP property for every non-trivial normal subgroup of $G$.
Thus  the congruence kernel of $\U(\Z G)$ is infinite if and only if $G$ has a CSP-critical epimorphic image.
The original problem hence reduces to classify the CSP-critical finite groups.
Assume that $G$ is a finite CSP-critical group.
Then $G$ is isomorphic to a subgroup of an exceptional simple component of $\Q G$.
Hence $G$ is a subgroup of either a division algebra or a two-by-two matrix over a division algebra (see the paragraph after Corollary~\ref{CSKFinite}).
The finite subgroups of division algebras have been classified by Amitsur \cite{Amitsur} and the finite subgroup of two-by-two matrices of division algebras have been classified by Banieqbal \cite{baniq}.
Thus, in order to classify the CSP-critical groups it suffices to decide which of the groups of  the classifications of Amitsur and Banieqbal are CSP-critical.

Unfortunately this  strategy encounters a serious difficulty.
Namely, the problem of deciding which algebras of type (EC1) have the CSP property seems to be far from reachable with the known techniques.
So we address a more modest problem which is an approximation to the problem of classifying the CSP-critical groups.
We say that a finite group is \emph{CSP'-critical} if $\Q G$ has an exceptional component but $\Q (G/N)$ has not exceptional components for any non-trivial normal subgroup $N$ of $G$.

The main theorem of this paper is the following (see notation in Section~\ref{SectionNotation}):

%\ChAngel{MAURICIO, HE ORDENADO LOS GRUPOS DE FORMA QUE PRIMERO EST\'{E}N LOS QUE SE INCLUYEN EN ANILLOS DE DIVISION Y NO SON Z-GRUPOS, DESPU\'{E}S LOS Z-GRUPOS, Y DESPU\'{E}S LOS DEM\'{A}S. LOS \'{U}LTIMOS ES SEGURO QUE SON SSP-CR\'{I}TICOS PERO LOS PRIMEROS HABR\'{I}A QUE ESTUDIARLOS POR SEPARADO. EN REALIDAD ALGUNOS DE ELLOS S\'{I} QUE SON SSP-CR\'{I}TICOS. POR EJEMPLO, SL(2,3) Y SL(2,5), YA QUE SU COMPONENTE EXCEPCIONAL ES DEL TIPO 2. SIN EMBARGO NO PODEMOS DECIR LO MISMO PARA LOS DE TIPOS \ref{Q8Cp} Y \ref{Zgrupos2}.}

\begin{theorem}\label{Main}
A finite group $G$ is CSP'-critical if and only if $G$ is isomorphic to one of the following groups.

\begin{enumerate}

% Z-GRUPOS

\item\label{Zgrupos1} $C_q\times (C_p\rtimes_2 C_4)$ with $p$ and $q$ different primes such that $3\ne p\equiv -1 \mod 4$, $q>2$ and $2\nmid o_q(p)$.

%\item\label{Zgrupos} $C_p\rtimes_{n_1} C_n$ with $p$ an odd prime different from $5$ and not dividing $n$, $\frac{n}{n_1}$ divisible by every prime dividing $n$, if $n$ is a power of $2$ then $4n_1\mid n$, and one of the following conditions hold: \ChMau{HAY UN ERROR EN LA CONDICION $4n_1 \mid n$ CREO QUE SOLO ES NECESARIA PARA EL CASO C, ABAJO EEXPLICO PORQUE}
%\begin{enumerate}
%\item $n_1=\gcd(n,p-1)$ and either $n$ is odd or $p\equiv 1 \mod 4$.
%\item $n_1=\gcd(n,p^2-1)$, $n$ is even and $p\equiv -1 \mod 4$.
%\item $n$ is even, $p\equiv -1 \mod 4$, $v_2(n_1)=1$, $v_2(n)>2$ and $v_q(n_1)=1$ and $ q^2 \nmid p-1$ for every prime divisor $q\neq 2$ of $n$.
%\end{enumerate}

\item\label{Zgrupos2} $C_p\rtimes_k C_n$ with $n\ge 8$, $p$ an odd prime not dividing $n$, $\frac{n}{k}$ is divisible by all the primes dividing $n$  and one of the following conditions holds:
    \begin{enumerate}
    \item $k=\gcd(n,p-1)$, either $n$ is odd or $p \equiv 1 \mod 4$ and if $p=5$ then $n=8$.
    \item $k=\gcd(n,p-1)$, $p\equiv -1 \mod 4$ and $n\ne 4$ and $v_2(n)=2$.
    \item $3\ne p\equiv -1 \mod 4$, $n=2^{v_2(p+1)+2}$ and $k=2^{v_2(p+1)+1}$.
    \end{enumerate}

% INCLUIDOS EN ANILLOS DE DIVISI\'{O}N QUE NO SON Z-GRUPOS

\item\label{Q8Cp}
$\q_8 \times C_p$ with $p$ an odd prime and $o_p(2)$ odd.

\item\label{SL3}  $\SL(2,3)=\GEN{i,j}_{\q_8}\rtimes \GEN{g}_3$, with $i^g=j$, $j^g=ij$.

\item\label{SL5} $\SL(2,5)%=\GEN{x,y,z | x^3=y^5=z^2=1, z=(xy)^2,(x,z)=(y,z)=1}
=\GEN{u,v\mid u^4=v^3=1,(uv)^5=u^2}$.

%\ChAngel{METABELIANOS NO INCLUIDOS EN ANILLOS DE DIVISI\'{O}N}

\item\label{D6}  $\D_6$.

\item\label{D8}  $\D_8$.

\item \label{G16-1}  $\D_{16}^+$.

\item \label{G16-2} $\q_8\rtimes C_2 = \GEN{i,j}_{\q_8}\rtimes \GEN{a}_2$, with $i^a=i\inv$ and $j^a=j$.

\item  \label{G24}$\q_8\times C_3$.

\item \label{G32} $\q_8 \Ydown_2 \D_8$.%=\GEN{i,j,a,b|i^4=i^2j^2=a^4=b^2=1,i^j=i\inv,a^b=a\inv,(i,a)=(j,a)=(i,b)=(j,b)=1,i^2=a^2}$.

\item\label{C5C8} $C_5 \rtimes_2 C_8$.

\item \label{G72} $(C_3\times C_3)\rtimes_2 C_8=(\GEN{a}_3\times \GEN{b}_3)\rtimes_2 \GEN{c}_8$, with $a^c=b^{-1}$ and $b^c=a$.

\item  \label{G96} $\SL(2,3)\Ydown_2 \D_8
%=
%    (\GEN{i,j}_{\q_8} \rtimes \GEN{a}_3)\Ydown_2 (\GEN{a}_4\rtimes \GEN{b}_2)$, with $i^a=j,j^a=ij$ and $a^b=a\inv
$.
%\GEN{i,j,g}_{C_2\times \SL(2,3)}\rtimes \GEN{z}_2=\GEN{i, j, g, z |  i^4=1, i^2=j^2, i^j=i^{-1}, g^6=1, i^g=j, j^g=ij, z^2=1, g^z=i^2g, (i,z)=(j,z)=1}$.
%    \ChAngel{TAL VEZ ME GUSTA M\'{A}S LA SIGUIENTE PRESENTACI\'{O}N}
%    $$(\GEN{i,j}_{\q_8}\times \GEN{a}_2)\rtimes \GEN{b}_6, \text{ with } i^b=j,j^b=ij,a^b=i^2b.$$
%\ChAngel{AUNQUE EN REALIDAD ESTE GRUPO ES $\SL(2,3)\Ydown_2 \D_8$.}

\item \label{G384} $\mathcal{C}=(\q_8\times \q_8)\rtimes C_6 = (\GEN{i_1,j_1}_{\q_8}\times \GEN{i_2,j_2}_{\q_8})\rtimes \GEN{b}_6$ with
$$i_1^b=i_2, \quad j_1^b=j_2, \quad i_2^b=j_1 \; \text{ and } \; j_2^b=i_1j_1.$$
%(\GEN{i,j}_{\q_8}\times \GEN{c,d}_{\q_8})\rtimes \GEN{a}_6$, with $i^a=c^3d$, $j^a=v^3$, $c^a=x$ and $d^a=x^3yv^2$.\ChAngel{Esto hay que arreglarlo}

%$\GEN{i, j, c, d, e, f |  i^4=1, i^2=j^2, i^j=i^{-1},  c^4=1, c^2=d^2, c^d=c^{-1}, (i,c)=(i,d)=(j,c)=(j,d)=1, e^3=1, i^e=i^{-1}j, j^e =i^{-1}c^2, c^e=i^2d, d^e=i^2c^{-1}d, f^2=1, i^f=i^2d, j^f=i^2c^{-1}, c^f=i^2 j c^2, d^f=ic^2, e^f=e}$

%\ChMau{PRIMITIVOS NO METABELIANOS}

%\item $\D_6$.

\item \label{G48} $\SL(2,3)\Ydown_2 C_4$.

%\item  $C_5\rtimes C_8$. \ChAngel{ESTE DEBE SER EL MISMO QUE EL DE \ref{C5C8}. ?`O NO?}

\item\label{SL9} $\SL(2,9)$.

%$ [320, 1581]=\GEN{i, j, a, b , c, d | i^4=1, i^2=j^2, i^j=i^{-1}, a^4=1, b^2=1, a^b=a^{-1}, i^2=a^2, (i,a)=(i,b)=(j,a)=(j,b)=1, c^5=1, i^c=i^2j, j^c=i^{-1}b, a^c=jba, b^c=ja, d^2=i^2, i^d=i^2a, j^d=ib, a^d=i^{-1}, b^d=ja, c^d=c^{-1}}$

\item \vspace{-8pt}\hspace{-8pt}
\label{G240-1} %$\SL(2,5)\nsp_4 \GEN{d_1}_8$, with $\SL(2,5)$ as in (\ref{SL5}), $x^{d_1}=x^{y^2x^2y}, y^{d_1}=y^3$ and $d_1^2=(xy)^{x^{y^2}}$.
$\matriz{{rcl}
\\
\A^+ = \SL(2,5)\nsp_4 C_8&=&\GEN{v,d|d^8=v^3=1,(d^2v)^5=d^4,v^{d}=vd^{-2}(v,d^2)}\\
&=&\GEN{u,v}_{\SL(2,5)}\nsp_4\GEN{d}_8, 
}$
\\ with $v^d=vu\inv(v,u)$ and $d^2=u$.

\item \vspace{-8pt}\hspace{-8pt}
\label{G240-2} %$\SL(2,5)\nsp_4 \GEN{d_1}_8$, with $\SL(2,5)$ as in (\ref{SL5}), $x^{d_1}=x^{y^2x^2y}, y^{d_1}=y^3$ and $d_1^2=(xy)^{x^{y^2}}$.
$\matriz{{rcl}
\\
\A^- = \SL(2,5)\nsp_4 C_8&=&\GEN{v,d|d^8=v^3=1,(d^{-2}v)^5=d^4,v^{d}=vd^2(v,d^{-2})}\\
&=&\GEN{u,v}_{\SL(2,5)}\nsp_4\GEN{d}_8, 
}$
\\ with $v^d=vu\inv(v,u)$ and $d^2=u^{-1}$. 
%$\SL(2,5)\GEN{d_1}=\GEN{x,y,z,d_1 | x^3=y^5=z^2=1, z=(xy)^2, (x,z)=(y,z)=1, x^{d_1}=x^{y^2x^2y}, y^{d_1}=y^3, d_1^2=(x^{y^2})^{-1}(xy)(x^{y^2})}$.

%\item \label{G240-2}$\SL(2,5)\GEN{d_2}=\GEN{x,y,z,d_2 | x^3=y^5=z^2=1, z=(xy)^2, (x,z)=(y,z)=1, x^{d_2}=x^{y^2x^2y}, y^{d_2}=y^3,  zd_2^2=(x^{y^2})^{-1}(xy)(x^{y^2})}$

\item \label{G160} $\B_1=(\q_8 \Ydown_2 \D_8)\rtimes C_5$ with with $\q_8=\GEN{i,j}$, $\D_8=\GEN{a}_4\rtimes\GEN{b}_2$, $C_5=\GEN{c}_5$,
    $$i^c=j^{-1}b, \quad j^c=i^{-1}, \quad a^c=ia^{-1}b  \quad \text{ and } \quad b^c=i^{-1}a^{-1}.$$

%$[160,199]=\GEN{i, j, a, b, c | i^4=1, i^2=j^2, i^j=i^{-1}, a^4=1, b^2=1, a^b=a^{-1}, i^2=a^2, (i,a)=(i,b)=(j,a)=(j,b)=1, c^5=1, i^c=j^{-1}b, j^c=i^{-1}, a^c=ia^{-1}b, b^c=i^{-1}a^{-1}}$

\item\label{G320} $\B_2=(\q_8\Ydown_2 \D_8) \nsp_2 (C_5\rtimes_2 C_4)$, with
	$\q_8=\GEN{i,j}$, $\D_8=\GEN{a}_4\rtimes \GEN{b}_2$, $C_5\rtimes_2 C_4=\GEN{c}_5\rtimes_2 \GEN{d}_4,$
    $$\matriz{{ccccccccccccccc}  i^c&=&i^2j, & j^c&=&i^{-1}b, & a^c&=&jba, & b^c&=&ja, \\
		i^d&=&i^2a, & j^d&=&ib, & a^d&=&i^{-1}, & b^d&=&ja & \text{ and }\; d^2&=&i^2.}$$

\item \label{G1920} $\B=(\q_8 \Ydown_2 \D_8)\nsp_2 \SL(2,5)$ with $\q_8=\GEN{i,j}$, $\D_8=\GEN{a}_4\rtimes\GEN{b}_2$, $\SL(2,5)=\GEN{u,v}$ as in (\ref{SL5}),  %where the kernel of the action is $Z(\SL(2,5))$.
    $$\matriz{{ccccccccccccccc} i^u&=&i^3, & j^u&=&jb,& a^u&=&i^3ab,& b^u&=&b, \\
    i^v&=&i^2jab,&j^v&=&i^3jab,&a^v&=&ib,&b^v&=&ia & \text{ and } \; u^2&=&i^2.}
    $$

%Presentation: $(\q_8 \Ydown_2 \D_8)=\GEN{i,j,a,b}$ as in \ref{G32}, $\SL(2,5)=\GEN{u,v|u^4=v^3=(u^2,v)=1,(uv)^5=u^2}$ and
%$$u^2=i^2, i^u=i^{-1}, j^u=jb, a^u=i^{-1}a^{-1}b, (b,u)=1,$$ $$ i^v=j^{-1}a^{-1}b, j^v=j^{-1}i^{-1}a^{-1}b, a^v=ib, b^v=i^{-1}a^{-1}$$
\end{enumerate}

%\newpage
%
%    $$\matriz{{c}u^2=i^2, \\ i^u=i^{-1}, j^u=jb, a^u=i^3ab, (b,u)=1, \\
%    i^v=jab, j^v=ijab, a^v=ib, b^v=ia.}$$
%
%$$\matriz{{c}
%u^2=i^2,\\
%i^u=i^3,j^u=jb,a^u=i^3ab,b^u=b, \\
%i^v=i^2jab,j^v=i^3jab,a^v=ib,b^v=ia, \\
%h^2=1, \\
%j^h=i^3j,i^h=i^3,a^h=a,b^h=i^2b, \\
%u^h=u^3,v^h=u^{vu^3}v^2}
%$$
%
%
%\begin{verbatim}
%x^4,y^2*x^2,x^y*x,
%x^2*a^2,b^2,a^b*a,
%Comm(x,a),Comm(x,b),Comm(y,a),Comm(y,b),
%u^2*x^2,v^3,(u*v)^5*u^2,
%x^u*(x*b)^-1,y^u*(x^2*y)^-1,a^u*(x^2*y*a*b)^-1,b^u*b^-1,
%x^v*(x^3*y*a*b)^-1,y^v*(x^3*a*b)^-1,a^v*(y*b)^-1,b^v*(y*a)^-1,
%h^2,
%x^h*(x^3*y)^-1,y^h*(x^2*y)^-1,a^h*a^-1,b^h*(x^2*b)^-1,
%u^h*(u^3)^-1,v^h*(u^(v*u^3)*v^2)^-1
%\end{verbatim}
%
%$$\matriz{{c}
%x^4=1,y^2=x^2,x^y=x\inv, \\
%x^2=a^2,b^2=1,a^b=a\inv, \\
%(x,a)=(x,b)=(y,a)=(y,b)=1, \\
%u^2=x^2,v^3=1,(uv)^5=u^2, \\
%x^u=xb,y^u=x^2y,a^u=x^2yab,b^u=b, \\
%x^v=x^3yab,y^v=x^3ab,a^v=yb,b^v=ya, \\
%h^2=1, \\
%x^h=x^3y,y^h=x^2y,a^h=a,b^h=x^2b, \\
%u^h=u^3,v^h=u^(vu^3)v^2}
%$$
%

\end{theorem}

We now discuss how far is the classification of CSP'-critical groups given by Theorem~\ref{Main} from the desired classification of CSP-critical groups and the consequences of Theorem~\ref{Main} to the original problem of classifying the finite groups $G$ for which the congruence kernel of $\U(\Z G)$ is finite.
If the congruence kernel of $\U(\Z G)$ is not finite then $\Q G$ has an exceptional component and therefore $G$ has a CSP'-critical epimorphic image. Therefore we at least have the following.

\begin{corollary}\label{CSKFinite}
Let $G$ be a finite group. If $G$ has not an epimorphic image isomorphic to any of the groups listed in Theorem~\ref{Main} then the congruence kernel of $\U(\Z G)$ is finite.
\end{corollary}

Suppose that $G$ is CSP'-critical and let $A$ be an exceptional component of $\Q G$. Let $\pi:\Q G \rightarrow A$ be a surjective homomorphism of algebras and let $N=\{g\in G:\pi(g)=1\}$. Then $A$ is an exceptional simple component of $\Q(G/N)$. Thus, by assumption, $N=1$ and hence $G$ is a subgroup of $A$.
Thus, if $G$ is CSP'-critical then $G$ can be embedded in any of its exceptional components.
If one of this exceptional components is of type (EC2) then $\Q G$ has not the CSP property and hence $G$ is CSP-critical.
Along the proof of Theorem~\ref{Main} we will show that the only groups in the list of Theorem~\ref{Main} not having an exceptional component of type (EC2) are precisely those of types (\ref{Zgrupos1})-(\ref{Q8Cp}) (Proposition~\ref{CSPDivision}).
So we have

\begin{corollary}\label{CriticalSeguros}
Let $G$ be a finite group.
\begin{enumerate}
 \item If $G$ is isomorphic to one of the groups in items (\ref{SL3})-(\ref{G1920}) of Theorem~\ref{Main} then $G$ is CSP-critical.
 \item If $G$ has an epimorphic image isomorphic to one of the groups in items (\ref{SL3})-(\ref{G1920}) of Theorem~\ref{Main} then the congruence kernel of $\U(\Z G)$ is infinite.
\end{enumerate}
\end{corollary}

Assume that $G$ is a finite CSP-critical group  other than the groups is items (\ref{SL3})-(\ref{G1920}) of Theorem~\ref{Main}.
Then $\Q G$ has not any exceptional component of type (EC2) and therefore $G$ is a subgroup of a division algebra.
This is the case of the groups of types (\ref{Zgrupos1})-(\ref{Q8Cp})  in Theorem~\ref{Main}.
For example if $G=\q_8\times C_p$ as in (\ref{Q8Cp}) then the only exceptional component of $\Q G$ is isomorphic to the quaternion algebra $\HQ(\Q(\zeta_p))$. In this case $G$ is a CSP-critical if and only if $\HQ(\Q(\zeta_p))$ has not the CSP property.
So to decide the question for these groups one need to  decide whether this algebra has the CSP property.
Similarly, if $G=C_q\times (C_p\rtimes_2 C_4)$ as in (\ref{Zgrupos1}) (respectively, $G\cong C_p\rtimes C_n$  as in (\ref{Zgrupos2})) then the only exceptional component of $\Q G$ is isomorphic to the quaternion algebra $A=\quat{(\zeta_p-\zeta_p\inv)^2,-1}{\Q(\zeta_q,\zeta_p+\zeta_p\inv)}$ (respectively, the cyclic algebra $A=(\Q(\zeta_{pk})/F(\zeta_k),\zeta_k)$, where $F$ is the only subfield of $\Q(\zeta_p)$ with $[\Q(\zeta_p):F]=\frac{n}{k}$). In both cases $G$ is CSP-critical if and only if $A$ has not the CSP property.
In case these groups are CSP-critical, then  the CSP-critical groups are exactly the CSP'-critical groups, i.e. those listed in Theorem~\ref{Main}. However if some of these groups is not CSP-critical then there could exists some  CSP-critical groups, not included in Theorem~\ref{Main}.
Such groups should have one proper epimorphic image isomorphic to one of the groups in items (\ref{Zgrupos1})-(\ref{Q8Cp}) of Theorem~\ref{Main}.

\section{Notation, preliminaries and some tools}\label{SectionNotation}

In this section we fix the notation which will be used throughout the paper.

The cardinality of a set $X$ is denoted by $\lvert X \rvert$. As it is customary, the Euler totient function will be denoted $\varphi$.

For $r,m$ and $p$ integers with $p$ prime and $\gcd(r,m)=1$ let
\begin{center}
\begin{tabular}{rl}
$v_p(m)\;=$ &maximum non-negative integer $k$ such that $p^k$ divides $m$; \\
$o_m(r)\;=$ &multiplicative order of $r$ module $m$, \\ & i.e. the minimum positive  integer $k$ such that $r^k \equiv 1 \mod m$; \\
$\zeta_m\;=$ &complex primitive $m$-th root of unity.
\end{tabular}
\end{center}

%\ChAngel{NOTACI\'{O}N CONJUNTISTA:}

%\ChAngel{NOTACI\'{O}N DE ANILLOS:}
%By default all the rings are associative and unital and the group of units of a ring $R$ is denoted $\U(R)$.

%\ChAngel{NOTACI\'{O}N COM\'{U}N PARA GRUPOS Y ANILLOS:}
We use the standard group and ring theoretical notation.
For example, the center of a group or ring $X$ is denoted $Z(X)$; if $a \in X$ and $Y\subseteq X$ then  $C_Y(a)$ denotes the centralizer of $a$ in $Y$ and we use the exponential notation for conjugation: $a^b = b\inv a b$.

%\ChAngel{NOTACI\'{O}N DE GRUPOS:}
If $G$ is a group, then $G'$ denotes the commutator subgroup of $G$ and $\exp(G)$ the exponent of $G$. If $g\in G$ then $\lvert g \rvert$ denotes the order of $g$ and $g^G$ denotes the conjugacy class of $g$ in $G$.
If $X\subseteq G$ then $\GEN{X}$ denotes the subgroup generated by $X$. This is simplified to $\GEN{g_1,\dots,g_n}$ for $X=\{g_1,\dots,g_n\}$.
Sometimes, we write $\GEN{g}_n$ to emphasize that $g$ has order $n$ or $\GEN{g_1,\dots,g_n}_G$ to represent a group isomorphic to $G$ and generated by $g_1,\dots,g_n$ (Theorem~\ref{Main} contains some examples of this).
If $H$ is a subgroup of $G$ then $N_G(H)$ denotes the normalizer of $H$ in $G$ and $\core_G(H)=\cap_{g\in G} H^G$, the core of $H$ in $G$.
The notation $H\le G$ (respectively, $H<G, H\unlhd G, H\lhd G$) means that $H$ is a subgroup (respectively, proper subgroup, normal subgroup, proper normal subgroup) of $G$.
If $p$ is a prime integer the $O_p(G)$ denotes the unique maximal normal $p$-subgroup of $G$.

We use the following constructions of groups.
\begin{eqnarray*}
G\rtimes_m H &=& \text{ semidirect product of $H$ acting on $G$ with kernel of order $m$}; \\
G\Ydown_m H &=& \text{ central product of $G$ and $H$ with  subgroups of order  m  identified}; \\
G \nsp_m H &=& \text{ semidirect product of $H$ acting on $G$ with subgroups of order } m \\ && \text{ identified}.
\end{eqnarray*}

Some groups that we encounter in the paper are
\begin{eqnarray*}
C_n &=& \text{ cyclic group of order } n; \\
\D_{2m}&=& \GEN{a}_m \rtimes \GEN{b}_2 \text{ with } a^b=a^{-1} \text{ (dihedral group of order }2m);\\
\q_{4m}&=& \GEN{j}_{2m}\nsp_2 \GEN{i}_4 \text{ with } j^i=j^{-1} \text{ (quaternion group of order } 4m); \\
\end{eqnarray*}
\begin{eqnarray*}
\D_{16}^+&=& \GEN{a}_8 \rtimes \GEN{b}_2, \text{ with } a^b=a^5. \\
\D_{16}^-&=& \GEN{a}_8 \rtimes \GEN{b}_2, \text{ with } a^b=a^3. \\
S_m &=& \text{ symmetric group on } m \text{ symbols;} \\
A_m &=& \text{ alternating group on } m \text{ symbols;} \\
\SL(n,q) &=& \{a\in M_n(\mathbb{F}_q) : \det(a)=1\}, \text{ with } \mathbb{F}_q\text{ the field with } q \text{ elements};\\
\PSL(n,q) &=& \SL(n,q)/Z(\SL(n,q));\\
%\bi^{*}&=&\GEN{s,t|(st)^2=s^3=t^5}, \text{binary icosahedral group}; \\
T^{*}_\alpha&=&\q_8\rtimes_{3^{\alpha-1}} \GEN{g}_{3^{\alpha}}
\text{ (observe that } T^{*}_1\cong \SL(2,3)).
%\text{ with }\footnote{$\GEN{i,j,g: i^4=1, i^2=j^2, i^j=i^{-1}, g^{3^{\alpha}}=1 i^g=j, j^g=ij}$} %\alpha\ge 1, i^g=j, j^g=ij;
\end{eqnarray*}

We also will encounter the following metacyclic groups
    \begin{equation}\label{Gmr}
    G_{m,r}=\GEN{a,b | a^m=1, b^n=a^t,a^b=a^r}=\GEN{a}_m\nsp_s \GEN{b}_{ns} =  \GEN{a^s}_t\rtimes_s \GEN{b}_{ns}
    \end{equation}
with
    \begin{equation}\label{GmrCond}
    \gcd(m,r)=1, n=o_m(r), s=gcd(r-1,m), st=m, \text{ and } \gcd(ns,t)=1.
    \end{equation}

%In order to describe some groups we use the following notation
%\begin{eqnarray*}
%\GEN{x_1,\dots,x_n}_{G} &=& \text{ group generated by } x_1,\dots,x_n \text{ and isomorphic to } G.
%\end{eqnarray*}

Let $F$ be a field of characteristic different of 2 and let $a,b$ non-zero elements of $F$. Then $\quat{a,b}{F}$ denotes the quaternion $F$-algebra $F[i,j|ji=-ij,i^2=a,j^2=b]$. Moreover $\HQ(F)=\quat{-1,-1}{F}$. A \emph{totally definite quaternion algebra} is a quaternion algebra of the form $\quat{a,b}{F}$ with $F$ a totally real number field and $a$ and $b$ totally negative, i.e. for every embedding $\sigma:F\rightarrow \C$, $\sigma(F)\subseteq \R$ and $\sigma(a)$ and $\sigma(b)$ are negative.

%\ChAngel{PRODUCTOS CRUZADOS:}
If $R$ is a ring and $G$ is a group then $R*^{\alpha}_{\tau} G$  denotes a \emph{crossed product} with action $\alpha:G\rightarrow \Aut(R)$ and twisting $\tau:G\times G \rightarrow \U(R)$ \cite{Passman89}, i.e. the associative ring $R*^{\alpha}_{\tau} G=\bigoplus_{g\in G} R u_g$ with multiplication given by the following rules:
	$$u_g a = \alpha_g(a) u_g \quad \text{and} \quad u_g u_h = \tau(g,h) u_{gh}, \quad (a\in R, g,h\in G).$$
In case $G=\GEN{g}_n$  then the crossed product $R*^{\alpha}_{\tau} G$ is completely determined by $\sigma=\alpha_g$ and $a=u_g^n$. In this case we follow the notation of \cite{Reiner75} and denote the crossed product by $(R,\sigma,a)$.
 A \emph{classical crossed product} is a crossed product $L*^{\alpha}_{\tau} G$, where $L/F$ is a finite Galois extension, $G = \Gal(L/F)$ and $\alpha$ is the natural action of $G$ on $L$. A classical crossed product $L *^{\alpha}_{\tau} G$ is denoted by $(L/F,\tau)$ \cite{Reiner75}.
A cyclic algebra is a classical crossed product $(L/F,\tau)$ where $\Gal(L/F)$ is  cyclic.
If $\Gal(L/F)=\GEN{\sigma}_n$ and $a=u_{\sigma}^n$ then the cyclic algebra $(L/F,\tau)$ is usually denoted $(L/F,a)$.
Every classical crossed product $(L/F,\tau)$ is a central simple $F$-algebra \cite[Theorem~29.6]{Reiner75}.

Consider a finite group $G$ with a cyclic normal subgroup $A=\GEN{a}_m$  and assume that $A$ is also a maximal abelian subgroup of $G$.
Fix a right inverse $\phi:G/A\rightarrow G$ of the natural projection $G\rightarrow G/A$ (i.e. $\phi(gA)A=gA$ for every $g\in G$).
Then we define a crossed product
    $$\Q(G,A)=\Q(\zeta_m)*^{\alpha}_{\tau} G/A,$$
with action and twisting given by
    \begin{eqnarray*}
    \alpha_{gA}(\zeta_m) &=& \zeta_m^i, \mbox{ if } a^{\phi(gA)}=a^iA \\
    \tau(gA,g'A) &=& \zeta_m^j, \mbox{ if }  \phi(gg'A)\inv\phi(gA)\phi(g'A)=a^j,
    \end{eqnarray*}
for each $g,g'\in G$. (Notice that $G/A$ is abelian because $A$ is the kernel of the action of the $G$ on $A$ by conjugation.) The algebra $\Q(G,A)$ is independent of the map $\phi$ up to isomorphisms.
Using the natural isomorphism $\Aut(A)\rightarrow \Aut(\Q(\zeta_m))$ one can transfer the action of $G$ on $A$ by conjugation to an action of $G$ on $\Q(\zeta_n)$. If $F$ is the fixed field of this action then $gA\mapsto \alpha_{gA}$ defines an isomorphism $G/A\rightarrow \Gal(\Q(\zeta_m)/F)$ and if we see this isomorphism as an identification then $\Q(G,A)$ is the classical crossed product $(\Q(\zeta_m)/F,\tau)$.

%\ChAngel{DESCOMPOSICI\'{O}N DE WEDDERBURN:}
We will need to compute the Wedderburn decomposition of $\Q G$ for some finite groups. For that we use the method introduce in \cite{ORS} which was extended in \cite{Olteanu} and implemented in the GAP package wedderga \cite{Wedderga}. We introduce the main lines of this method now (see \cite{ORS} for details).

Let $G$ be a finite group. For a subgroup $H$ of $G$, let $\widehat{H}=\frac{1}{|H|}\sum_{h\in H} h$. Clearly, $\widehat{H}$ is an idempotent of $\Q G$ which is central if and only if $H$ is normal in $G$. If $K\lhd H\leq G$  and $K\neq H$ then let
    $$\varepsilon(H,K)=\prod (\widehat{K}-\widehat{M})=\widehat{K}\prod (1-\widehat{M}),$$
where $M$ runs through the set of all minimal normal subgroups of $H$ containing $K$ properly. We extend this notation by setting $\varepsilon(H,H)=\widehat{H}$. Clearly $\varepsilon(H,K)$ is an idempotent of the group algebra $\Q G$. Let $e(G,H,K)$ be the sum of the distinct $G$-conjugates of $\varepsilon(H,K)$, that is, if $T$ is a right transversal of $C_G(\varepsilon(H,K))$ in $G$, then $$e(G,H,K)=\sum_{t\in T}\varepsilon(H,K)^t.$$ Clearly, $e(G,H,K)$ is a central element of $\Q G$ and if the $G$-conjugates of $\varepsilon(H,K)$ are orthogonal, then $e(G,H,K)$ is a central idempotent of $\Q G$.

A \emph{strong Shoda pair} of $G$ is a pair $(H,K)$ of subgroups of $G$ satisfying the following conditions: $K\leq H\unlhd N_G(K)$, $H/K$ is cyclic and a maximal abelian subgroup of $N_G(K)/K$ and the different $G$-conjugates of $\varepsilon(H,K)$ are orthogonal.

If $(H,K)$ is a strong Shoda pair of $G$ then $H$ has a linear character with kernel $K$, which we denote $\lambda_{H,K}$.
Moreover $\lambda_{H,K}^G$, the character of $G$ induced by $\lambda_{H,K}$, is irreducible, $\ker \lambda_{H,K}^G=\core_G(K)$
and $e(G,H,K)$ is the a unique primitive central idempotent $e$ of $\Q G$ with $\lambda_{H,K}^G(e)\ne 0$.
If $N=N_G(K)$ and $n=[G:N]$ then
	\begin{equation}\label{SSPComponent}
	 \Q G \; e(G,H,K) \cong M_n(\Q(N/K,H/K))
	\end{equation}
%where
%\begin{eqnarray*}
%k&=&[H:K],\\
%N&=&N_G(K),\\
%\alpha_{nH}(\zeta_k) &=& \zeta_k^i, \mbox{ if } yK^{\phi(nH)}=y^iK \mbox{ and}\\
%\tau(nH,n'H) &=& \zeta_k^j, \mbox{ if }  \phi(nn'H)\inv\phi(nH)\phi(n'H)=y^jK,
%\end{eqnarray*}
%for $nH,n'H\in N/H$ and integers $i$ and $j$.
A simple component of $\Q G$ of the form $\Q Ge(G,H,K)$, for $(H,K)$ a strong Shoda pair is called an \emph{SSP component} of $\Q G$.

A group is said to be \emph{strongly monomial} if all the simple components of $\Q G$ are SSP components.
Every abelian-by-supersolvable groups is strongly monomial \cite{ORS}.
The following theorem shows that for metabelian groups we can compute the primitive central idempotents of $\Q G$ using some special strong Shoda pairs.

\begin{theorem}\label{SSPmetabelian}\cite{ORS}
Let $G$ be a metabelian finite group and let $A$ be a maximal abelian subgroup of $G$ containing $G'$. The primitive central idempotents of $\Q G$ are the elements of the form $e(G,H,K)$, where $(H,K)$ is a pair of subgroups of $G$ satisfying the following conditions:
\begin{enumerate}
\item \label{metabelian1}$H$ is a maximal element in the set $\{B\leq G \mid A\leq B \mbox{ and } B'\leq K\leq B\}$;
\item \label{metabelian2}$H/K$ is cyclic.
\end{enumerate}
\end{theorem}

The classification of the finite groups which are subgroups of division rings was obtained by Amitsur \cite{Amitsur}.
If $G$ is a finite subgroup of a division ring then a Sylow subgroup of $G$ is either cyclic or a quaternion 2-group.
A \emph{Z-group} is a finite subgroup of a division ring with all Sylow subgroup cyclic. For the readers convenience we include in the following theorem the classification of finite subgroups of division rings in the form presented in \cite{shirvani}.

\begin{theorem}\cite{Amitsur,shirvani}\label{SubgruposAD}

\begin{enumerate}
 \item[(Z)] The Z-groups are
	\begin{enumerate}
	 \item the finite cyclic groups,
	 \item $C_m \rtimes_2 C_4$ with $m$ odd and $C_4$ acting by inversion on $C_m$ and
%	 \item $G_0\times G_1 \times \dots \times G_s$ with $|G_0|,|G_1|,\dots,|G_s|$ pairwise relatively prime, where $G_0$ is cyclic, $G_i=C_{p^a}\rtimes_{q_1^{\alpha_1}\cdots q_k^{\alpha_k}} C_{q_1^{\beta_1}\cdots q_k^{\beta_k}}$ with $p,q_1,\dots,q_k$ different primes and for every $j=1,\dots,k$ both $v_{q_j}(o_{\frac{|G|}{|G_j|}}(p))<o_{q_j^{\alpha_j}}(p)$ and one of the following conditions holds:
%	 \item $C_{n_0}\times (C_{p_1^{\alpha_1}}\rtimes_{m_1} C_{n_1}) \times \dots \times (C_{p_k^{\alpha_k}}\rtimes_{m_k} C_{n_k})$ with $n_0,n_1,\dots,n_k,p_1,\dots,p_k$ pairwise relatively prime and for every $i=1,\dots,k$, and every prime divisor $q$ of $n_i$ we have
%    $$p_i \text{ is prime}, \quad v_q(m_i)<v_q(n_i), \quad v_{q}(o_{n_0\prod_{j\ne i} p_j^{\alpha_j}n_j}(p_i))<o_{q^{v_q(m_i)}}(p_i)$$
%and one of the following conditions holds:
%	\begin{enumerate}
%	 \item $v_{q}(p_i-1) \le v_q(m_i)$ and either $q\ne 2$ or $q=2$ and $p_i\equiv 1 \mod 4$.
%	 \item $q=2$, $p_i\equiv -1 \mod 4$, and $v_2(m_i)$ is either $1$ or greater than $v_2(p_i+1)$.
%	\end{enumerate}
    \item $C_m\rtimes_k C_n$ with $\gcd(m,n)=1$ and, using the following notation
    \begin{eqnarray*}
    P_p&=& \text{ Sylow } p\text{-subgroup of } C_m,\\
    Q_p&=& \text{ Sylow } p\text{-subgroup of } C_n,\\
    X_p&=&\{q\mid n : q \text{ prime and } (P_p,Q_q)\ne 1\},\\
    R_p&=&\prod_{q\in X_p} Q_q;
    \end{eqnarray*}
    we have $C_n=\prod_{p\mid m} R_p$ and the following properties hold for every prime $p\mid m$ and $q\in X_p$:
        \begin{enumerate}
        \item $v_{q}\left(o_{\frac{\vert G\vert}{\vert P_p\vert \; \vert R_p\vert }}(p)\right)<o_{q^{v_q(k)}}(p)$,
        \item if $q$ is odd or  $p\equiv 1 \mod 4$ then $v_{q}(p-1) \le v_q(k)$ and
        \item if $q=2$ and $p\equiv -1 \mod 4$ then $v_2(k)$ is either $1$ or greater than $v_2(p+1)$.
	   \end{enumerate}
    \end{enumerate}
 \item[(NZ)]  The finite subgroups of division rings which are not Z-groups are
	\begin{enumerate}
	\item $\bo^{*}=\GEN{s,t|(st)^2=s^3=t^4}$  (binary octahedral group),
	\item $\q_{m}$ with $v_2(m)\ge 3$.
  \item $\q_8\times M$ with $M$ a $Z$-group of odd order  such that $o_{|M|}(2)$ is odd,
	\item $\SL(2,3)\times M$, with $M$ a $Z$-group of order coprime to 6 and $o_{|M|}(2)$  odd, and
  \item $\SL(2,5)$.
	\end{enumerate}
\end{enumerate}
\end{theorem}

Following the proof of Theorem~\ref{SubgruposAD}, in either \cite{Amitsur} or \cite{shirvani}, one can discover a minimal division ring containing each of the groups in the classification.
We will need this for the non-abelian Z-groups.
Assume $G=\GEN{a}_m\rtimes_k \GEN{b}_n$ with $\gcd(m,n)=1$, $b^a=a^r$ and $n=o_m(r)$.
Then $A=\GEN{a,b^{\frac{n}{k}}}$ is cyclic and normal and maximal abelian in $G$ and hence $(A,1)$ is a strong Shoda pair of $G$.
Moreover, $a\mapsto \zeta_m$, $b\mapsto u_b$ determines an injective group homomorphisms $G\rightarrow \U(\Q(G,A))$.
Furthermore, $\Q(G,A)$ is the algebra given by the presentation $\Q(\zeta_{mk})[u_b | \zeta_m^{u_b}=\zeta_m^r, u_b^{\frac{n}{k}}=\zeta_k]$.
If $f:G\rightarrow \U(D)$ is an injective group homomorphism then $f(a)$ and $f(b^{\frac{n}{k}})$ are commuting roots of unity of order $m$ and $k$ respectively and $f(a^b)=f(a)^r$.
Therefore $\zeta_m\mapsto f(a)$ and $u_b\mapsto f(b)$ defines an algebra homomorphism $\Q(G,A)\rightarrow D$, which is injective because $\Q(G,A)$ is simple.
In particular, if $D$ is a division algebra then so is $\Q(G,A)$.
This proves the following:

\begin{lemma}\label{DivisionAlgebraZGrupo}
Let $G=\GEN{a}_m\rtimes_k \GEN{b}_n$ with $\gcd(m,n)=1$, $b^a=a^r$ and $n=o_m(r)$ and let $A=\GEN{a,b^{\frac{n}{k}}}$.
Then $G$ is a subgroup of a division algebra if and only if $\Q(G,A)$ is a division algebra.
\end{lemma}

\section{Sufficiency}

In this section we prove the sufficiency part of Theorem~\ref{Main}, namely we prove that all the groups listed in the theorem are CSP'-critical.
We have to prove that for each $G$ in the list of Theorem~\ref{Main}, $\Q G$ has an exceptional component while $\Q \overline{G}$ has not exceptional components for any  proper epimorphic image $\overline{G}$ of $G$.
Table~\ref{InfoCSP} displays the exceptional component obtained in each case and, except for the first three infinite families, identify the groups in the GAP library.

\begin{table}[h]
$$\matriz{{llllll}
\# & G & \text{Excep. Comp.} & %\text{NA Quotients} &
    \text{GAP ID} \\\hline

(\ref{Zgrupos1}) & C_q\times (C_p\rtimes_2 C_4) & \quat{(\zeta_p-\zeta_p\inv)^2,-1}{\Q(\zeta_q,\zeta_p+\zeta_p\inv)} & %C_p\rtimes_2 C_4, &
    \\
%& q\ge 3,p\ge 5, \text{ primes} & & %C_q\times \D_p, \D_p
%\\

(\ref{Zgrupos2}) & C_p\rtimes_k C_n & (\Q(\zeta_{pk})/F(\zeta_k),\zeta_k) & %C_p\rtimes_h C_{\frac{hn}{k}}
    \\
& & ([\Q(\zeta_p):F]=\frac{n}{k})  & %(k\ne h\mid k)
    \\

(\ref{Q8Cp}) & \q_8 \times C_p, & \HQ(\Q(\zeta_p)) & %\q_8
\\

(\ref{SL3}) & \SL(2,3) & M_2(\Q(\zeta_3)) & %A_4 &
    [24,3] \\

(\ref{SL5}) & \SL(2,5) & M_2\quat{-1,-3}{\Q} & % A_5 &
    [120,5] & \\

(\ref{D6}) & \D_6 & M_2(\Q) & [6,1]\\

(\ref{D8}) & \D_8 & M_2(\Q) & [8,3] \\

(\ref{G16-1}) & \D_{16}^+ & M_2(\Q(\zeta_4)) & [16,6] \\

(\ref{G16-2}) & \q_8\rtimes C_2 & M_2(\Q(\zeta_4)) & [16,13] \\

(\ref{G24}) & \q_8\times C_3 & M_2(\Q(\zeta_3)) & % \q_8 &
    [24,11] \\

(\ref{G32}) & \q_8 \Ydown_2 \D_8 & M_2(\HQ(\Q)) & [32,50] \\

(\ref{C5C8}) & C_5 \rtimes_2 C_8 & (\Q(\zeta_5)/\Q,-1) & %C_5\rtimes_1 C_4 &
    [40,3] \\

(\ref{G72}) & (C_3\times C_3)\rtimes_2 C_8, & M_2\quat{-1,-3}{\Q} & % [36,9] &
[72,19] \\
% & a^c=b\inv,b^c=a \\

%(\ref{D6}) & \D_8 & M_2(\Q) & & [8,3] \\
(\ref{G96}) & \SL(2,3)\Ydown_2 \D_8 & M_2(\HQ(\Q)) & % A_4, C_2\times A_4, &
    [96,202] \\
% & & & C_2^2\times A_4  \\

(\ref{G384}) & \mathcal{C} & M_2(\HQ(\Q)) & % A_4, C_2\times A_4,  &
    [384,618] \\
%           &                              &              &  G/Z(G) \\
%           &                              &              & G/(Z(\q_8\times \q_8))\\

(\ref{G48}) & \SL(2,3)\Ydown_2 C_4 & M_2(\Q(\zeta_4)) & % A_4, C_2\times A_4 &
    [48,33] \\

(\ref{SL9}) & \SL(2,9) & M_2\quat{-1,-3}{\Q} & % \PSL(2,9) &
[720,409] & \\

%
%\ChMau{NUEVOS PRIMITIVOS}
%
% $\D_6$.
%
%
%  $C_5\rtimes C_8$.
%
%
%
%\ChMau{NOTE QUE HAY ALGUNOS REPETIDOS}
%
%\ChMau{PRIMITIVOS NO RESOLUBLES}
%
(\ref{G240-1}) &\A^+ & (\Q(\zeta_5)/\Q,-1) & % S_5 &
    [240,90] \\

(\ref{G240-2}) &\A^- & (\Q(\zeta_5)/\Q,-1) & % S_5 &
[240,89] \\

%
% (\ref{G240-2}$\SL(2,5)\GEN{d_2}=\GEN{x,y,z,d_2 | x^3=y^5=z^2=1, z=(xy)^2, (x,z)=(y,z)=1, x^{d_2}=x^{y^2x^2y}, y^{d_2}=y^3,  zd_2^2=(x^{y^2})^{-1}(xy)(x^{y^2})}$
(\ref{G160}) & \B_1 & M_2(\HQ(\Q)) & % [80,49] &
    [160,199] \\
 (\ref{G320}) & \B_2 & M_2(\HQ(\Q)) & % D_{10}, C_4 \rtimes D_{10} &
    [320, 1581] \\
(\ref{G1920}) & \B & M_2(\HQ(\Q)) & % A_5, C_2^5\rtimes A_5 &
    [1920,241003]

}$$
\caption{\label{InfoCSP} The list of CSP'-critical groups. The third column displays the exceptional component of the rational group algebra. The last column represents the identification of the group in the GAP library of small groups \cite{GAP} except for the first three families of groups.}
\end{table}

To prove that the groups of type (\ref{Zgrupos2}) are CSP'-critical we need the following lemma.

\begin{lemma}\label{SemidirectoComponenteUnica}
Let $G=\GEN{a}_p\rtimes_{k} \GEN{b}_n$ with $p$ prime not dividing $n$.
Let $A=\GEN{a,b^{n/k}}$ and let $F$ be the only subfield of $\Q(\zeta_p)$ of degree $\frac{(p-1)k}{n}$.
Then the non-commutative simple components of $\Q G$ are the algebras of the form $B_h=\Q G \; e(G,A,\GEN{b^{hn/k}})$ with $h\mid k$. Moreover $B_h \cong (\Q(\zeta_{ph})/F(\zeta_h),\zeta_h)$ for every $h\mid k$ .
\end{lemma}

\begin{proof}
As $G'=\GEN{a}$ and $A$ is cyclic and maximal abelian in $G$, a pair of subgroups of $G$ satisfying the conditions of Theorem~\ref{SSPmetabelian}  is either of the form $(G,K)$ with $G'\subseteq K$ or of the form $(A,K)$ with $K\cap G'=1$.
If $G'\subseteq K$ then $\Q Ge(G,G,K)$ is commutative. Hence the lemma follows from Theorem~\ref{SSPmetabelian} and (\ref{SSPComponent}).
\end{proof}

\begin{proposition}\label{Suficiencia}
If $G$ is one of the groups listed in Theorem~\ref{Main} then $G$ is CSP'-critical. Moreover $\Q G$ has an exceptional component of type (EC1) if and only if $G$ is of one of the types (\ref{Zgrupos1})-(\ref{Q8Cp}) and $\Q G$ has an exceptional component of type (EC2) if and only if $G$ is of one of the types (\ref{SL3})-(\ref{G1920}).
\end{proposition}

\begin{proof}
(\ref{Zgrupos1}) Assume that $G=C_q\times (C_p\rtimes_2 C_4)=\GEN{a}_{pq}\rtimes_2 \GEN{b}_4$ satisfies the conditions of (\ref{Zgrupos1}).
Let $A=\GEN{a,b^2}$, a cyclic subgroup of index 2 in $G$.

Using Theorem~\ref{SSPmetabelian} one can calculate the non-commutative simple components of $\Q G$. They are
\begin{eqnarray*}
 A_1 &=& \Q G e(G,A,1) \cong \Q(\zeta_{pq}/\Q(\zeta_q,\zeta_p+\zeta_p\inv),-1) \cong \quat{(\zeta_p-\zeta_p\inv)^2,-1}{\Q(\zeta_q,\zeta_p+\zeta_p\inv)}, \\
 A_2 &=& \Q G e(G,A,\GEN{b^2}) \cong \Q(\zeta_{pq}/\Q(\zeta_q,\zeta_p+\zeta_p\inv),1) \cong M_2(\Q(\zeta_q,\zeta_p+\zeta_p\inv)), \\
 A_3 &=& \Q G e(G,A,\GEN{a^p}) \cong \Q(\zeta_{p}/\Q(\zeta_p+\zeta_p\inv),-1) \cong \quat{(\zeta_p-\zeta_p\inv)^2,-1}{\Q(\zeta_p+\zeta_p\inv)} \text{ and} \\
 A_4 &=& \Q G e(G,A,\GEN{a^p,b^2}) \cong \Q(\zeta_{p}/\Q(\zeta_p+\zeta_p\inv),1) \cong M_2(\Q(\zeta_p+\zeta_p\inv)).
\end{eqnarray*}
Observe that $m=pq$, $n=4$ and $k=2$ satisfy the conditions of (Z)(c) in Theorem~\ref{SubgruposAD}.
Thus $G$ is a subgroup of a division algebra and hence, by Lemma~\ref{DivisionAlgebraZGrupo}, $A_1=\Q(G,\GEN{a,b^2})$ is a division ring.
It is not totally definite because its center contains a primitive root of unity of order $q$. Thus $A_1$ is an exceptional algebra of type (EC1).
However $A_3$ is a totally definite quaternion algebra because $(\zeta_p-\zeta_p\inv)^2<0$ and $A_2$ and $A_4$ are not exceptional because $[\Q(\zeta_p+\zeta_p\inv):\Q]=\frac{p-1}{2}>2$.
On the other hand every non-abelian proper quotient of $G$ is isomorphic to either $C_p\rtimes C_4$, $\D_{2p}$ or $C_q\times \D_{2p}$. $\Q(C_p \rtimes C_4)$ has two non-commutative components isomorphic to $A_3$ and $A_4$ respectively, the only non-commutative component of $\D_{2p}$ is isomorphic to $A_4$ and the non-commutative components of $\Q(C_q\times \D_{2p})$  are isomorphic to $A_2$ and $A_4$ respectively.
Thus $\Q \overline{G}$ has not exceptional components for any proper quotient $\overline{G}$ and we conclude that $G$ is CSP'-critical.

(\ref{Zgrupos2})
Assume that $G=C_p\rtimes_k C_n=\GEN{a}_p\rtimes_k \GEN{b}_n$ with $p$, $n$ and $k$ as in (\ref{Zgrupos2}), and let $A=\GEN{a,b^{\frac{n}{k}}}$.
The non-commutative simple components of $\Q G$ are the algebras  $B_h$ of Lemma~\ref{SemidirectoComponenteUnica} with $h\mid k$.
The degree of $B_h$ is $\frac{n}{k}$.
Moreover, $B_h\cong \Q(H_h,A_h)$ with $H_h=G/\GEN{b^{h\frac{n}{k}}}\cong C_p\rtimes_h C_{\frac{hn}{k}}$ and $A_h=A/\GEN{b^{h\frac{n}{k}}}$.
By Lemma~\ref{DivisionAlgebraZGrupo}, $B_h$ is a division algebra if and only if $H_h$ is one of the groups in items (Z)(b) or (Z)(c) in Theorem~\ref{SubgruposAD}.
Observe that $G$ satisfies the conditions of (Z)(c) in Theorem~\ref{SubgruposAD}.
Hence $B_k$ is a division algebra.
If $B_k$ is a totally definite quaternion algebra then $\frac{n}{k}=2$ and $k=2$ because the centre of $B_k$ has a root of unity of order $k$.
Hence $n=4$ in contradiction with the hypothesis. Thus $B_k$ is an exceptional algebra of type (EC1).

Any non-abelian simple component of the rational group algebra of a proper quotient $\overline{G}$ of $G$ is isomorphic to $B_h$ for some proper divisor $h$ of $k$.
So, in order to prove that $\Q \overline{G}$ has not exceptional components, it is enough to prove that if $h$ is a proper divisor of $k$ then $B_h$ is not exceptional.
Assume first that $B_h$ is exceptional of type (EC2). Then $\frac{n}{k}$, which is the degree of $B_h$, is $2$ or $4$ and hence $n$ is a power of 2.
Moreover the centre of $B_h$ contains the unique subfield $F$ of index $\frac{n}{k}$ in $\Q(\zeta_p)$.
Thus either $\frac{n}{k}=2$ and $F=\Q(\zeta_p+\zeta_p\inv)$ or $\frac{n}{k}=4$, $F=\Q$ and $p=5$.
The latter case is in contradiction with the hypothesis and in the former case $F$ is contained in an imaginary quadratic extension of $\Q$ and hence $F=\Q$ and $p=3$, again in contradiction with the hypothesis.
This proves that $B_h$ is not exceptional of type (EC2) and in particular that $\Q G$ has not exceptional components of type (EC2).
Secondly suppose that $B_h$ is exceptional of type (EC1).
Then $B_h$ is a division algebra containing $H_h$ and hence $H_h$ is a non-cyclic Z-group.
Thus $H_h=C_p\rtimes_h C_{\frac{hn}{k}}$  satisfies either (Z)(b) or (Z)(c).
This does not hold if $p\equiv 1 \mod 4$ because in that case $v_q(h)<v_q(k)=v_q(p-1)$ for some prime $q\mid \frac{hn}{k}$.
Similarly, if $\frac{k}{h}$ is divisible by an odd prime $q$ then $v_q(h)<v_q(k)=v_q(p-1)$ and hence $H_h$ is not a $Z$-group.
Thus $p\equiv -1 \mod 4$ and $\frac{k}{h}$ is a power of 2.
If $G$ satisfies (\ref{Zgrupos2})(a) or (\ref{Zgrupos2})(b) then $v_2(k)=1$ and hence $h$ is odd.
Then $H_h$ does not satisfies the conditions of neither (Z)(b) nor (Z)(c), a contradiction.
Thus $G$ satisfies (\ref{Zgrupos2})(c). As $v_2(h)<v_2(k)=v_2(p+1)+1$, $H_h$ only can satisfy the conditions of (Z)(b) or (Z)(c) if $h=2$.
In this case $\frac{hn}{k}=4$, so that $H_h=C_p\rtimes_2 C_4$. Then $B_h=(\Q(\zeta_p)/\Q(\zeta_p+\zeta_p\inv),-1)$, a totally definite quaternion algebra and hence it is not an exceptional component, as desired.

(\ref{Q8Cp})  Let $G=\q_8\times C_p$ with $p$ an odd prime such that $o_p(2)$ is odd.
Then the only non-commutative simple components of $\Q G$ are $\HQ(\Q)$ and $\HQ(\Q(\zeta_p))$. The first one is a totally definite quaternion algebra and the second one is a division algebra by \cite{Moser1973}.
Since $\Q(\zeta_p)$ is not totally real $\HQ(\Q(\zeta_p))$ is exceptional of type (EC1) and $\Q G$ has not exceptional components of type (EC2).
The only non-abelian proper quotient of $G$ is $G/C_p\cong \q_8$ and the only non-abelian simple component of $\Q \q_8$ is $\HQ(\Q)$. This proves that $G$ is CSP'-critical.

To prove that if $G$ is one of the groups in items (\ref{SL3})-(\ref{G1920}) of Theorem~\ref{Main} has an exceptional component of type (EC2) we will simply calculate the Wedderburn decomposition of their rational group algebras and will observe that all of them have one component of isomorphic to one of the following algebras:
	\begin{equation}\label{Excepcionales}
	 \matriz{{c} M_2(\Q), M_2(\Q(\zeta_4)), M_2(\Q(\zeta_3)), M_2(\Q(\sqrt{-2})), \\ M_2(\HQ(\Q)),M_2\quat{-1,-3}{\Q}, (\Q(\zeta_5)/\Q,-1).}
	\end{equation}
The only one which is not obviously exceptional of type (EC2) is $A=(\Q(\zeta_5)/\Q,-1)$.
Clearly $A$ is not a division algebra, because its exponent is the order of $-1$ modulo the image of the norm of $\Q(\zeta_5)$ over $\Q$ \cite[Corollary~30.7]{Reiner75}.
So it is enough to prove that $\R \otimes_{\Q} A$ is not split.

Observe that $A=\Q(G,\GEN{a})=\Q Ge(G,\GEN{a},1)$ with $G=\GEN{a}_5\rtimes_2 \GEN{b}_8$.
The only irreducible character of $G$ not vanishing in $e(G,\GEN{a},1)$ is $\chi=\lambda_{\GEN{a},1}^G$ and it is
 given by the second and third row of the following table.
	$$\matriz{{cccccccccc}
	1 & 5 & 5 & 1 & 4 & 5 & 5 & 5 & 4 & 5 \\
	1 & b &  b^2 &  b^4 &  a &  b^3 &  b^5 &  b^6 &  ab^4 &  b^7 \\
	4 & 0 & 0 & -4 & -1 & 0 & 0 & 0 & -1 & 0
	}$$
The first row gives the cardinality of the corresponding conjugacy class.
Having in mind that $\chi(a^2)=\chi(a)$, because $a$ and $a^2$ are conjugate in $G$, we can calculate the Frobenius-Schur indicator of $\chi$ which is
	\begin{eqnarray*}
	 \frac{1}{|G|}\sum_{g\in G} \chi(g^2) &=&
	\frac{1}{20}(\chi(1) + 5(\chi(b^2) +  \chi(b^4) + \chi(b^6))+  4\chi(a^2))=-1.
	\end{eqnarray*}
Therefore, $\R\otimes_{\Q} A$ is not split as desired. Thus $A=M_2(D)$ with $D$ a totally definite quaternion algebra over $\Q$.
(It can be proved that $D\cong\quat{-2,-5}{\Q}$, but this is not relevant for us.)

% \ChAngel{(ESTA P\'{A}GINA EST\'{A} MUY FEA. ANTES DE SOMETERLO HAY QUE MANIPULAR LA LISTA DE DESCOMPOSICIONES DE WEDDERBURN DE ABAJO PARA QUE CORTE EN EL LUGAR QUE NOS INTERESE. PERO ESO S\'{O}LO LO SABREMOS CUANDO TODO LO DEM\'{A}S SEA DEFINITIVO)}

To compute the Wedderburn component of the groups algebras of the groups in items (\ref{SL3})-(\ref{G1920}) we use the Wedderga package \cite{Wedderga} and obtain the following decompositions:
	\begin{eqnarray*}
	\Q\; \SL(2,3) &=& \Q \oplus \Q(\zeta_3) \oplus M_3(\Q) \oplus \HQ(\Q) \oplus M_2(\Q(\zeta_3)).\\
	\Q\; \SL(2,5) &=&
  \Q  \oplus M_4(\Q) \oplus M_5(\Q) \oplus M_3(\Q(\sqrt{5})) \oplus M_3(\HQ( \Q)) \oplus \\
                && (\Q(\zeta_5)/\Q(\sqrt{5}),-1)\oplus M_2 \quat{-1,-3}{\Q}. \\
	\Q \D_6 & = & 2\Q \oplus M_2(\Q). \\
	\Q \D_8 & = & 4\Q \oplus M_2(\Q). \\
	\Q \D_{16}^+ &=& 4\Q \oplus  2\Q(\zeta_4)\oplus M_2(\Q(\zeta_4)). \\
	\Q (\q_8\rtimes C_2) &=& 8\Q \oplus  M_2(\Q(\zeta_4)). \\
	\Q (\q_8 \times C_3) &=& 4\Q \oplus  4\Q(\zeta_3)\oplus \HQ(\Q)\oplus M_2(\Q(\zeta_3)). \\
	\Q (\q_8\Ydown_2 \D_8) &=& 16\Q \oplus  M_2(\HQ(\Q)). \\
	\Q (C_5 \rtimes_2 C_8) &=& 2 \Q \oplus \Q(\zeta_4) \oplus \Q(\zeta_8) \oplus M_4(\Q) \oplus (\Q(\zeta_5)/\Q,-1). \\
	\Q (C_3^2 \rtimes_2 C_8) &=& 2\Q \oplus \Q(\zeta_4)\oplus \Q(\zeta_8)\oplus 2 M_4(\Q)\oplus 2M_2\quat{-1,-3}{\Q}. \\
	\Q (\SL(2,3)\Ydown_2 \D_8) &=& 4\Q \oplus 4\Q(\zeta_3)\oplus 4 M_3(\Q)\oplus
                \\ && M_2\left(\left(\Q(\zeta_{12})/\Q(\zeta_3),-1\right)\right)\oplus M_2(\HQ(\Q)). \\
	\Q \mathcal{C} &=& 2\Q \oplus 2\Q(\zeta_3)\oplus 2 M_3(\Q)\oplus 2 M_4(\Q)\oplus 2 M_6(\Q)\oplus 2 M_4(\Q(\zeta_3))\\
&& \oplus M_6(\HQ(\Q))\oplus M_2((\Q(\zeta_{12})/\Q(\zeta_3),-1)) \oplus M_2(\HQ(\Q)). \\
	\Q(\SL(2,3)\Ydown_2 C_4) &=& 2\Q \oplus 2\Q(\zeta_3)\oplus 2 M_3(\Q)\oplus M_2(\Q(\zeta_{12}))\oplus M_2(\Q(\zeta_4)). \\
	\Q\; \SL(2,9) &=& \Q \oplus 2 M_5(\Q) \oplus M_9(\Q) \oplus M_{10}(\Q) \oplus M_8(\Q(\sqrt{5})) \oplus
							\\ && M_4\quat{-1,-3}{\Q(\sqrt{5})} \oplus  M_4\quat{-1,-3}{\Q(\sqrt{2})}\oplus 2 M_2\quat{-1,-3}{\Q}. \\
	\Q \A^+ &=& 2\Q\oplus 2 M_4(\Q)\oplus 2 M_5(\Q)\oplus M_6(\Q)\oplus M_4(\Q(\zeta_3))\oplus \\
					&& M_6(\Q(\sqrt{-2}))\oplus (\Q(\zeta_5)/\Q,-1).\\
\Q \A^- &=& 2\Q\oplus 2 M_4(\Q)\oplus 2 M_5(\Q)\oplus M_6(\Q) \oplus (\Q(\zeta_5,\sqrt{3})/\Q(\sqrt{3}),-1) \\
					&&\oplus  M_3(\Q(\sqrt{2})\otimes_{\Q} D)) \oplus (\Q(\zeta_5)/\Q,-1), \\
					&& \text{where } (\Q(\zeta_5)/\Q,-1)=M_2(D). \\
	\Q \B_1 &=& \Q \oplus \Q(\zeta_5)\oplus 3 M_5(\Q)\oplus M_2((\Q(\zeta_{20})/\Q(\zeta_5),-1))\oplus M_2(\HQ(\Q)). \\
	\Q \B_2 &=& 2\Q \oplus  6 M_5(\Q)\oplus  M_2(\Q(\sqrt{5}))\oplus M_4(\HQ(\Q(\sqrt{5})))\oplus 2 M_2(\HQ(\Q)).\\
\Q\B&=&\Q\oplus  2M_5(\Q)\oplus 2M_{10}(\Q)\oplus M_{15}(\Q)\oplus M_5(\Q(\zeta_3))\oplus M_{20}(\Q)\oplus \\
	&&M_4(\Q)  \oplus M_3(\Q(\sqrt{5}))\oplus M_8(\HQ(\Q))\oplus M_{10}(\HQ(\Q))\oplus \\
	&& M_6(\HQ(\Q(\sqrt{5})))\oplus M_2(\HQ(\Q)).
	\end{eqnarray*}

As we anticipated above each decomposition has one component of one of the types in (\ref{Excepcionales}).
Moreover no simple component is exceptional of type (EC1).

To complete the proof it remains to prove that if $G$ is one of the groups in items (\ref{SL3})-(\ref{G1920}) of Theorem~\ref{Main} and $H$ is a proper epimorphic image of $G$ then $\Q H$ has not exceptional components.
Some of these groups $G$ have not any abelian proper quotient and hence there is nothing to prove for them.
The remaining groups $G$ and their proper quotients $H$ are listed in Table~\ref{NAPQ} (second and third columns respectively).
\begin{table}[h]
$$\matriz{{cll}
\# & G & H \\\hline

(\ref{SL3}) & \SL(2,3) & \PSL(2,3) \cong A_4 \\

(\ref{SL5}) & \SL(2,5) & \PSL(2,5) \cong A_5 \\

(\ref{G24}) & \q_8\times C_3 & \q_8 \\

(\ref{C5C8}) & C_5 \rtimes_2 C_8 & C_5\rtimes C_4 \\

(\ref{G72}) & (C_3\times C_3)\rtimes_2 C_8, &  (C_3\times C_3)\rtimes C_4 \\

(\ref{G96}) & \SL(2,3)\Ydown_2 \D_8 & A_4, C_2\times A_4\\

(\ref{G384}) & \mathcal{C} & A_4, C_2\times A_4, \mathcal{C}/Z(\mathcal{C}), \mathcal{C}/Z(\q_8\times \q_8)\\

(\ref{G48}) & \SL(2,3)\Ydown_2 C_4 & A_4, C_2\times A_4 \\

(\ref{SL9}) & \SL(2,9) & \PSL(2,9) \\

(\ref{G240-1}) &\A^+ & S_5 \\
(\ref{G240-2}) &\A^- & S_5 \\

(\ref{G160}) & \B_1 &  \B_1/Z(\B_1) \\
 (\ref{G320}) & \B_2 & \D_{10}, \B_2/Z(\B_2) \\

(\ref{G1920}) & \B & A_5, \B/Z(\B)=C_2^4\rtimes A_5
}$$
\caption{\label{NAPQ}
The non-abelian proper quotients of some of the CSP'-groups of types (\ref{SL3})-(\ref{G1920})}.
\end{table}
We have to check that $\Q H$ has not exceptional components for each $H$ in the third column of Table~\ref{NAPQ}.
We do not need to consider the groups $H$ which are a proper quotient of another group $K$ in the third column because the simple components of $\Q H$ are also simple components of $\Q K$. This excludes $A_4,A_5, \mathcal{C}/Z(\q_8\times \q_8)$ and $\D_{10}$. Moreover, $\Q (C_2\times H)\cong \Q C_2 \otimes_{\Q} \Q H = 2\Q H$ and hence do not need to consider the groups of the form $C_2\times H$.
Hence we have only to verify that the following Wedderburn decomposition have not exceptional components:
    \begin{eqnarray*}
%    \Q A_4 &=& \Q \oplus M_3(\Q).\\
    \Q \q_8 &=& 4\Q \oplus \HQ(\Q). \\
    \Q (C_5\rtimes C_4) &=& 2\Q \oplus \Q(\zeta_4) \oplus M_4(\Q). \\
    \Q ((C_3\times C_3)\rtimes C_4) &=&2\Q \oplus \Q(\zeta_4)\oplus 2 M_4(\Q).\\
    \Q (\mathcal{C}/Z(\mathcal{C})) &=& 2\Q \oplus 2\Q(\zeta_3)\oplus 2 M_3(\Q)\oplus 2 M_4(\Q)\oplus 2 M_6(\Q)\oplus 2 M_4(\Q(\zeta_3)). \\
    \Q \PSL(2,9) &=& \Q \oplus 2 M_5(\Q) \oplus M_9(\Q) \oplus M_{10}(\Q) \oplus M_8(\Q(\sqrt{5})). \\
    \Q S_5&=&2\Q\oplus 2 M_4(\Q)\oplus 2 M_5(\Q)\oplus M_6(\Q). \\
    \Q (\B_1/Z(\B_1) &=& \Q \oplus \Q(\zeta_5)\oplus 3 M_5(\Q). \\
    \Q (\B_2/Z(\B_2) &=& 2\Q \oplus  6 M_5(\Q)\oplus  M_2(\Q(\sqrt{5})).\\
    \Q (\B/Z(\B))&=&\Q\oplus  2M_5(\Q)\oplus 2M_{10}(\Q)\oplus M_{15}(\Q)\oplus M_5(\Q(\zeta_3))\oplus \\ && M_{20}(\Q)\oplus M_4(\Q)  \oplus M_3(\Q(\sqrt{5}))
    \end{eqnarray*}
This finishes the proof of the proposition.
\end{proof}

\section{Necessity}

In this section we prove the necessity part of Theorem~\ref{Main}.
More precisely, throughout we assume that $G$ is a finite CSP'-critical group and we prove that $G$ is one of the groups listed in Theorem~\ref{Main}.
Since we have already proved that the groups in the theorem are CSP'-critical and, in particular, they have an exceptional component, we have

\medskip
(P1) No proper quotient of $G$ is isomorphic to one of the groups in Theorem~\ref{Main}.
\medskip

As $G$ is CSP'-critical, $\Q G$ has at least one exceptional component. Let $A$ be any exceptional component of $\Q G$ and let $f:G\rightarrow \U(A)$ be the group homomorphism induced by the projection $\Q G\rightarrow A$. Then $A$ is a simple exceptional component of $\Q (G/ \ker f)$ and hence $f$ is injective. Thus

\medskip
(P2) $G$ is embedded in one (all) exceptional component(s) of $\Q G$.
\medskip

\begin{definition}
Let $G$ be a finite subgroup of $\GL_2(D)$ where $D$ is a division algebra of characteristic zero.
One says that $G$ is primitive in $M_2(D)$ if there is no non-singular $2\times 2$ matrix $A$ over $D$ such that $AgA^{-1}$ is monomial (i.e. every column and row has exactly one non-zero entry) for all $g\in G$.
Otherwise one says that $G$ is imprimitive.
\end{definition}

We consider separately the following cases:
\begin{itemize}
\item  $G$ is a subgroup of a division ring (Proposition~\ref{CSPDivision}).
\item $G$ is not a subgroup of a division ring and $G$ is metabelian (Proposition~\ref{CSPMetabelian}).
\item $G$ is not a subgroup of a division ring, $G$ is not metabelian and $G$ is an imprimitive subgroup of an exceptional component of $\Q G$ (necessarily of type (EC2)) (Proposition~\ref{CSPImprimitive}).
\item $G$ does not satisfy any of the previous conditions (Proposition~\ref{CSPPrimitive}). In particular, $G$ is not metabelian and it is a primitive subgroup of an exceptional component of type (EC2).
\end{itemize}

The following proposition proves Theorem~\ref{Main} for the case when $G$ is a subgroup of a division ring.

\begin{proposition}\label{CSPDivision}
Let $G$ be a CSP'-critical group which can be embedded in a division ring.
Then $G$ is isomorphic to one of the groups in items (\ref{Zgrupos1})-(\ref{SL5}) of Theorem~\ref{Main}.
%Moreover, $\Q G$ has an exceptional component of type (EC1) if and only if $G$ is one of the groups in items (\ref{Zgrupos1})-(\ref{Q8Cp}) in Theorem~\ref{Main}.
%\begin{enumerate}
%\item $\q_8 \times C_p$ with $p$ an odd prime such that $o_p(2)$ is odd.
%
%\item $C_p\rtimes C_n$ with $p$ an odd prime different from $3$ and $5$  and not dividing $n$ and if $n_1$ is the order of the kernel or the action of $C_n$ on $C_p$ then $\frac{n}{n_1}$ is divisible by every prime dividing $n$,  and one of the following conditions hold:
%\begin{enumerate}
%\item $n_1=\gcd(n,p-1)$ and either $n$ is odd or $p\equiv 1 \mod 4$.
%\item $n_1=\gcd(n,p^2-1)$ and $n$ is even and $p\equiv -1 \mod 4$.
%\item $n$ is even, $p\equiv -1 \mod 4$, $v_2(n) >2$,  $v_q(n_1)=1$ and $q^2\nmid p-1$ for every prime divisor $q$ of $n$.
%\end{enumerate}
%
%\end{enumerate}
\end{proposition}

\begin{proof}
By the hypothesis $G$ is one of the groups in Theorem~\ref{SubgruposAD}. Assume first that $G$ is not a Z-group, that is $G$ satisfies one of the conditions (a)-(e) of (NZ).
First, $G$ is not of type (NZ)(a) because $\bo^{*}=\GEN{s,t|(st)^2=s^3=t^4}$ and $\bo^*/\GEN{t^2,(s,t^2)}\cong \D_6$, in contradiction with (P1).
If $G$ is of type (NZ)(e) then $G\cong \SL(2,5)$, that is $G$ is as in item (\ref{SL5}) of Theorem~\ref{Main}.
By (P1), if $G$ is of type (NZ)(d) then $G\cong \SL(2,3)$, thus $G$ is as in item (\ref{SL3}) of Theorem~\ref{Main}.
%In Section~3 we have calculated the Wedderburn decompositions of $\Q \SL(2,3)$ and $\Q \SL(2,5)$.
%The only non-commutative division rings appearing simple components are $\HQ(Q)$ for $\SL(2,3)$ and $(\HQ(\zeta_5)/\Q(\sqrt{5}),-1)$. Both are totally definite quaternion algebras, so that they are not exceptional.

We now prove that $G$ is not of type (NZ)(b).
Suppose that $G=\q_{m}=\GEN{j}_{\frac{m}{2}}\nsp_2 C_4$ with $t=v_2(m)\ge 3$. 
If $t\ge 3$ then $G/\GEN{j^4}\cong \D_8$ in contradiction with (P1).
If $3\mid m$ then $G/\GEN{j^3}\cong \D_6$ in contradiction with (P1).
Hence $t=3$ and $3\nmid m$.
Moreover $A=\GEN{j}$ is a maximal abelian subgroup of index 2 in $G$ and $G'=\GEN{j^2}$.
Thus, by Theorem~\ref{SSPmetabelian}, the non-commutative simple components of $\Q G$ are of the form $\Q G e(G,A,\GEN{j^d})$ with $d\mid \frac{m}{2}$ and $d\ne 1,2$. Then
	$$\Q G e(G,A,\GEN{j^d})\cong \left\{\matriz{{ll}
	(\Q(\zeta_d)/\Q(\zeta_d+\zeta_d\inv),1) \cong M_2(\Q(\zeta_d+\zeta_d\inv)), & \text{if } d\mid \frac{m}{4}; \\
	(\Q(\zeta_d)/\Q(\zeta_d+\zeta_d\inv),-1) \cong \HQ(\Q(\zeta_d+\zeta_d\inv)), & \text{otherwise}.}\right.$$
None of this algebras is exceptional (the former because $d\ne 1,2,3,4$ or 6 and the latter because it is a totally definite quaternion algebra).
This yields a contradiction.

Assume that $G$ is of type (NZ)(c), i.e. $G=\q_8\times M$, with $M$ a Z-group of odd order $m>1$ and $o_m(2)$ is odd.
We claim that $m$ is prime.
As $G$ is a subgroup of a division algebra $D$, the subalgebra $D_1$ of $D$ generated by $M$ is a division algebra which appears in the Wedderburn decomposition of $\Q M$. As $M$ is a proper quotient of $G$, $D_1$ is not an exceptional component and hence the degree of $D_1$ is at most 2. However, since the order of $M$ is odd the degree of $D_1$ is odd. Thus $M$ is abelian. Then $M=C_n\times N$ for some subgroup $N$ and some integer $n>1$.
If $p$ is a divisor of $n$ then $M$ has a subgroup $K$ of index $p$ and $o_p(2)$ is odd. Then $G/K\cong \q_8\times C_p$. By (P1) we deduce that $K=1$ and hence $m=p$, as desired.
We conclude that $G=\q_8\times C_p$ with $p$ an odd prime such that $o_p(2)$ is odd. Thus $G$ is the group of item (\ref{Q8Cp}) of Theorem~\ref{Main}.
This finishes the case when $G$ is not a Z-group.

Suppose now that $G$ is a CSP'-critical Z-group.
Hence $G$ is of type (Z)(b) or (Z)(c). We first prove that $G$ is not of type (Z)(b).
Otherwise $G=\GEN{a}_m\rtimes_2 \GEN{b}_4$ with $m$ odd and $b$ acting by inversion.
If $3\mid m$ then $G/\GEN{a^3,b^2}\cong \D_6$, in contradiction with (P1). Thus $3\nmid m$.
Moreover, $A=\GEN{a,b^2}$ is a maximal abelian subgroup of $G$ of index 2 and $G'=\GEN{a}$.
Using Theorem~\ref{SSPmetabelian}, we have that every non-commutative simple component of $\Q G$ is of the form $\Q G e(G,A,B)$ with $B$ a subgroup of $A$ not containing $a$. Fix such $B$ and assume that $k=[A:B]$.
By (\ref{SSPComponent}), $\Q G \; e(G,A,B)=(\Q(\zeta_k)/\Q(\zeta_k+\zeta_k\inv),\epsilon)$ with $\epsilon=1$ if $b^2\in B$ and $\epsilon=-1$ if $b^2\not\in B$. In the second case $\Q G\; e(G,A,B)$ is a totally definite quaternion algebra.
In the first case $\Q G \; e(G,A,B)=M_2(\Q(\zeta_k+\zeta_k\inv))$ which is not exceptional because $k$ is divisible by a prime $p>3$.
Therefore $\Q G$ has not any exceptional component, contradicting the hypothesis.

It remains to consider type (Z)(c).
So suppose that $G=C_m\rtimes_k C_n$ satisfies the hypothesis of (Z)(c) and recall that the Sylow $p$-subgroup of $C_m$ is denoted $P_p$ and the Sylow $p$-subgroup of $C_n$ is denoted $Q_p$.
Let $p$ be a prime divisor of $m$ such that $C_n$ does not act trivially on $P_p$.
Let $q_1,\dots,q_h$ be the prime divisors $q$ of $n$ such that $Q_q$ acts non-trivially on $P_p$ and recall that $R_p=Q_{q_1}\cdots Q_{q_h}$.
Let $k_p$ be the order of the kernel of the action of $R_p$ on $P_p$.
Then $P_p\rtimes_{k_p} R_p$ is a direct factor of $G$ and a subgroup of a division algebra of degree $\frac{|R_p|}{k_p}$ (see Lemma~\ref{DivisionAlgebraZGrupo}).
This division algebra is an exceptional component unless $\frac{|R_p|}{k_p}=2$.
Therefore, since $G$ is CSP'-critical, if $G\ne P_p\rtimes_{k_p} R_p$ then $\frac{|R_p|}{k_p}=2$.
This only can happen for one prime $p$, because $Q_2$ acts non-trivially on at most one Sylow $p$-subgroup of $C_m$.
This shows that $C_m=C_{m_0}\times P_p$ and $G=C_{m_0}\times (P_p\rtimes_k C_n)$ with $k=k_p$, $C_n=R_p=\GEN{b}$ and $P_p=\GEN{a}_{p^{\alpha}}$.
Then $A=C_{m_0}\times \GEN{a,b^{n_p/k_p}}$ is a cyclic normal subgroup and a maximal abelian subgroup of $G$.
Hence $(A,1)$ is a strong Shoda pair of $G$ and hence $\Q (G,A,1)$ is a simple component of $\Q G$ isomorphic to $\Q (G,A)$.
Moreover, by Lemma~\ref{DivisionAlgebraZGrupo}, $\Q (G,A)$ is a division algebra of degree $\frac{n}{k}$.

We claim that $\alpha=1$. Otherwise $\overline{G}=G/\GEN{a^{p^{\alpha-1}}}$ satisfies the conditions of (Z)(c) and therefore
$\Q(\overline{G},\overline{A})$ is a division ring of degree $\frac{n}{k}$ which is a component of $\Q\overline{G}$.
By assumption, $\Q(\overline{G},\overline{A})$ has to be a totally definite quaternion algebra.
Thus $n=2k$, and the action of $\GEN{b}$ on $\GEN{a}$ is the only one of order 2, i.e. $a^b=a\inv$.
Moreover, as the centre of $\Q (\overline{G},\overline{A})$ contains $\Q(\zeta_{m_0k})$ and $m_0$ is odd, necessarily $m_0=1$ and $k\le 2$.
However, the hypothesis of (Z)(c) implies that $k\ne 1$ and thus $k=2$.
Then $G$ is as in case (Z)(b) which we have already excluded. This concludes the proof of the claim.

Suppose that $m_0\ne 1$.
Then $\GEN{a}_p\rtimes_{k} \GEN{b}_n$ is a non-commutative proper quotient of $G$ and its rational group algebra has a simple component of degree $\frac{n}{k}$.
The argument of the previous paragraph shows that $n=4$, $k=2$ and $a^b=a\inv$.
A similar argument, now factoring by a subgroup of $C_{m_0}$, shows that $m_0$ is prime.
Applying the conditions in (Z)(c) for $q=2$ we have $v_2(o_{m_0}(p))<o_2(p)=1$ or equivalently $o_{m_0}(p)$ is odd.
If $p=3$ then $\D_6$ is a proper epimorphic image of $G$ in contradiction with (P1).
Moreover, if $p\equiv 1 \mod 4$ then $1=v_2(k)\ge v_2(p-1)\ge 2$, another contradiction. Thus $3\ne  p\equiv -1 \mod 4$.
We have proved that $G$ satisfy the conditions of item (\ref{Zgrupos1}) in Theorem~\ref{Main} (with $q=m_0$).

Finally, suppose that $m_0=1$. Then $G=\GEN{a}_p\rtimes_k \GEN{b}_n$, $A=\GEN{a,b^{\frac{n}{k}}}$ and every Sylow subgroup of $\GEN{b}$ acts non-trivially on $\GEN{a}$.
Hence, if $q$ is a prime divisor of $n$ then $1 \le v_q\left(\frac{n}{k}\right) \le v_q(p-1) \le v_q(k)$.
In particular, $v_q(n)=2$ and hence either $n=4$ and $G=C_p\rtimes_2 C_4$ or $n\ge 8$. However in the first case $G$ satisfies (Z)(b) which was excluded before. Thus $n\ge 8$.
Suppose  that $p\equiv 1 \mod 4$. Then, by the hypothesis on (Z)(c), $v_q(p-1)\le v_q(k)$ for every prime divisor $q$ of $n$.
Therefore $\gcd(n,p-1)$ divides $k$.
By means of contradiction, suppose that $\gcd(n,p-1)\ne k$.
Then $v_q(p-1)<v_q(k)$ for some prime divisor $q$ of $n$.
Let $C$ be the subgroup of order $q$ of $\GEN{b}$ and $\overline{G}=G/C$.
Then $1\ne C\subseteq Z(G)$, $\overline{G}=C_p\rtimes_{\frac{k}{q}} C_{\frac{k}{q}}$, $(\overline{A},1)$ is a strong Shoda pair of $\overline{G}$ and $\Q(\overline{G},\overline{A})=\Q(G,A)$ is a simple component of $\Q \overline{G}$.
Then $\overline{G}$ also satisfies the assumptions of (Z)(c), since $v_q(p-1)\le v_q\left(\frac{k}{q}\right)$.
By Lemma~\ref{DivisionAlgebraZGrupo}, $\Q(G,A)$ is a non-commutative division algebra of degree $\frac{n}{k}$.
Therefore $n=2k$, the centre of this algebra is a totally real field containing $\Q(\zeta_k)$ and $k>2$ because $v_2(k)\ge v_2(p-1)\ge 2$.
This yields a contradiction.
Hence $k=\gcd(n,p-1)$.
If $p=5$ then $\frac{n}{k}$ divides $4$ and hence $n$ is a power of $2$. Then $k=4$ and $n$ is either $8$ or $16$.
If $n=16$ then $G$ has an epimorphic image isomorphic to $C_5\rtimes_2 C_8$ in contradiction with (P1).
Thus $n=8$ and $G$ satisfies the conditions of (\ref{Zgrupos2})(a) in Theorem~\ref{Main}.

Assume that $p\equiv -1 \mod 4$. Since $\frac{n}{k}$ divides $p-1$, if $n$ is even then $v_2(k)=v_2(n)-1$.
By the hypothesis of (Z)(c), $n\ne 2$ and, as $G$ does not satisfy (Z)(b), $n\ne 4$.
Moreover, $v_q(p-1)\le v_q(k)$ for every odd prime divisor $q$ of $n$ and if $2$ divides $n$ then either $v_2(k)=1$ or $v_2(k)>v_2(p+1)$.
If $v_q(p-1)<v_q(k)$ for some odd prime divisor of $n$ or $v_2(p+1)+1<v_2(k)$ then the argument of the previous paragraph yields a contradiction.
Thus $v_q(p-1)=v_q(k)$ if $q$ is an odd prime divisor of $n$ and if $n$ is even then $v_2(k)$ is either $1$ or $v_2(p+1)+1$.
If $n$ is odd then $k=\gcd(p-1,n)$ and $G$ satisfies the conditions of (\ref{Zgrupos2})(a) in Theorem~\ref{Main}.
Assume otherwise that $n=2^vn_1$, with $v\ge 1$ and $2\nmid n_1$. Then either $k=2\gcd(p-1,n_1)$ or $k=2^{v-1}\gcd(p-1,n_1)$ and $v=v_2(p+1)+2$.
In the second case $v\ge 4$ and $\overline{G}=G/\GEN{b^{2^{v-2}}}=C_p\rtimes_{2\gcd(p-1,n_1)} C_{4\gcd(p-1,n_1)}$ satisfies the conditions of (Z)(c).
Hence $\Q (\overline{G},\overline{A})$ is a division algebra of degree $2\gcd(p-1,n_1)$ in the Wedderburn decomposition of $\Q \overline{G}$.
By hypothesis $\Q (\overline{G},\overline{A})$ is a totally definite quaternion algebra. Then $\gcd(p-1,n_1)=1$ and $n_1\mid m_0=1$.
Hence either $k=2\gcd(p-1,n_1)$ with $n_1\ne 1$ or $n=2^{v_2(p+1)+2}$ or $n=4$. However the last case has been excluded above. The two remaining cases correspond to items (\ref{Zgrupos2})(b) and (\ref{Zgrupos2})(c) of Theorem~\ref{Main}.
\end{proof}

In order to prove Theorem~\ref{Main} for the metabelian groups not covered by Proposition~\ref{CSPDivision} we start describing the SSP exceptional components of type (EC2).

\begin{proposition}\label{componentesdessp}
Let $(H,K)$ be a strong Shoda pair of $G$ such that $\lambda_{H,K}^{G}$ is faithful and let $A=\Q G e(G,H,K)$ and $N=N_G(K)$.
If $A$ is an exceptional component of type (EC2) then one of the following conditions holds:
\begin{enumerate}
\item \label{Mat2Q}  $A\cong M_2(\Q)$ and $G\cong\D_m$ with $m=6,8$ or $12$.
\item \label{Mat2Q-1} $A\cong M_2(\Q(\zeta_4))$ and one of the following conditions holds:
    \begin{enumerate}
    \item Either $G\cong \D_{16}^+$, or $G\cong C_4\times \D_6$, or $G\cong C_3\rtimes_4 C_8$.
    \item $[H:K]=4$, $[G:H]=2$, $N=H$ and $|K|\in \{2,4\}$.
    \end{enumerate}
\item \label{Mat2Q-2} $A=M_2(\Q(\sqrt{-2}))$ and $G\cong \D_{16}^-$.
\item\label{Mat2Q-3} $A\cong M_2(\Q(\zeta_3))$ and one of the following conditions hold:
    \begin{enumerate}
    \item $G\cong C_3\times \D_8$ or $G\cong C_3\times \q_8$.
    \item $[G:H]=2$, $N=H$, $[H:K]=3$ or $6$ and $1\ne |K| \mid [H:K]$.
    \end{enumerate}
\item \label{Mat2-2-5Q} $A\cong (\Q(\zeta_5)/\Q,-1)$ and $G=C_5\rtimes_2 C_8$.
\item \label{Mat2HQQ} $A\cong M_2(\HQ(\Q))$ and one of the following conditions hold:
    \begin{enumerate}
    \item $G=\GEN{i,j}_{\q_{16}}\rtimes \GEN{a}_2$ with $j^a=j^3$ and $i^a=i$.
    \item $[G:N]=2$, $[H:K]=4$, $N/K\cong \q_8$ and $|K|\in \{2,4\}$.
    \end{enumerate}
\item \label{Mat2-1-3Q} $A\cong M_2\quat{-1,-3}{\Q}$, $[G:N]=2$, $[H:K]=6$, $N/K\cong \q_{12}$ and $1\ne |K|\mid [H:K]$.
\end{enumerate}
In particular, $|G|\in \{6,8,12,16,18,24,32,36,40,48,64,72,144\}$.
\end{proposition}

\begin{proof}
Let $n=[G:N]$, $k=[H:K]$,  and $F=Z(A)$.
By (\ref{SSPComponent}), $A\cong M_n(\Q (\zeta_k)*^{\alpha}_{\tau} N/H)$  and $\alpha$ is faithful. Hence the degree of $A$ is $[G:H]$ and $[\Q(\zeta_k):F]=[N:H]$.
Thus
    \begin{equation}\label{Phik}
    \varphi(k)=[\Q(\zeta_k):\Q]=[N:H][F:\Q].
    \end{equation}

Since $A$ is an exceptional component of type (EC2), either $[G:H]=2$ and $D=\Q$ or $[G:H]=2$ and $D$ is a quadratic imaginary extension of $\Q$, or $[G:H]=4$, $n\le 2$ and $D$ is a totally definite quaternion algebra over $\Q$.
Moreover $\core_G(K)=\ker \lambda_{H,K}^G=1$. So, if $K=1$ then $N=G\ne H$ and otherwise $[G:N]=2$ and $K\cap K^g=1$ for each $g\in G\setminus N$.
In both cases $N\unlhd G$. Thus $K^g\leq N$ and $|K|$ divides $[N:K]$.
%The same argument shows that if $H\unlhd G$ then $|K|$ divides $[H:K]$.
We consider separately the following cases:
\begin{itemize}
\item[(A)] $K=1$, $[G:H]=2$ and $F=\Q$. Then $G=N$ and $A=M_2(\Q)$.
\item[(B)] $K=1$, $[G:H]=2$ and $F\neq \Q$. Then $G=N$, $F$ is an imaginary quadratic extension of $\Q$ and $A=M_2(F)$.
\item[(C)] $K=1$ and $[G:H]=4$. Then $G=N$, $F=\Q$ and $A=M_2(D)$ with $D$ a totally definite quaternion algebra over $\Q$.
\item[(D)] $K\ne 1$ and $[N:H]=2$. Then $[G:N]=2$, $F=\Q$ and $A=M_2(D)$ with $D$ a totally definite quaternion algebra over $\Q$.
\item[(E)] $K\ne 1$, $N=H$ and $F=\Q$. Then $[G:H]=2$ and $A=M_2(\Q)$.
\item[(F)] $K\ne 1$, $N=H$ and $F\neq \Q$. Then $[G:H]=2$, $F$ is an imaginary quadratic extension of $\Q$ and $A=M_2(F)$.
\end{itemize}
%In the first three cases $G=N$ and in the last three cases $[G:N]=2$.
%Moreover in cases (C) and (D) $A=M_2(D)$, with $D$ a totally definite quaternion algebra over $\Q$.
%Finally, in cases (B) and (F), $F$ is an imaginary quadratic extension of $\Q$.

(A) In this case $H$ is a cyclic subgroup of order $k$ and index 2 in G. Moreover, $\varphi(k)=2$, by (\ref{Phik}) reads $\varphi(k)=2$. Hence either $G=\D_{2k}$ with $k=3,4$ or $6$ or $G=\q_{2k}$ with $k=4$ or $6$.
However the only non-commutative simple component of $\Q \q_8$ is isomorphic to $\HQ(\Q)$ and if $G=\q_{12}$ then $A=\quat{-1,-3}{\Q}$. ($\Q \q_{12}$ has also an SSP simple component $\Q \q_{12} e(\q_{12},\GEN{a},\GEN{a^3})\cong M_2(\Q)$ coming from the strong Shoda pair $(\GEN{a},\GEN{a^3})$
but $\lambda_{\GEN{a},\GEN{a^3}}^{\q_{12}}$ is not faithful.) Thus $G$ satisfies the conditions of (\ref{Mat2Q}).

(B) In this case $(\ref{Phik})$ reads $\varphi(k)=4$. Hence $k=5,10,8$ or $12$. Moreover, $F$ is an imaginary quadratic
extension of $\Q$ contained in $\Q(\zeta_k)$ and this excludes the cases $k=5$ and $10$. Therefore either $k=8$ and $F\subseteq \Q(\zeta_8)$  or $k=12$ and $F\subseteq \Q(\zeta_{12})$.
%Moreover $F$ is either $\Q(\sqrt{-2})$, $\Q(\zeta_4)$ or $\Q(\zeta_3)$ because the first two are the only imaginary quadratic subfields of $\Q(\zeta_8)$ and the last two are the only imaginary quadratic subfields of $\Q(\zeta_{12})$.
Let $a$ be a generator of $H$ and let $b\in G\setminus H$.
Then $G=\GEN{a,b}$,  $a^b=a^i$ and $b^2=a^j$ with $i=-1,5$ or $-5$ and $k\mid j(i-1)$.
However, if $i=-1$ then $F=\Q(\zeta_k+\zeta_k\inv)$, a real quadratic extension of $\Q$, contradicting the assumptions.

Assume that $k=8$. If $i=5$ then $F=\Q(\zeta_4)$, $j$ is even and $(a^{\frac{j}{2}}b)^2=1$.
Replacing $b$ by $a^{\frac{j}{2}}b$ one may assume that $b^2=1$. Then $G\cong \D_{16}^+$ and the conditions of (\ref{Mat2Q-1})(a) holds.
If $i=-5$ then $F=\Q(\sqrt{-2})$, $4\mid j$ and $(a^{\frac{j}{4}}b)^2=1$. Again one may assume that $b^2=1$.
Now $G\cong \D_{16}^-$ and the conditions of (\ref{Mat2Q-2}) holds.

Assume that $k=12$. Suppose that $i=5$. Then $F=\Q(\zeta_4)$ and $3\mid j$. If $6\mid j$ then, $(a^{-\frac{j}{6}}b)^2=1$ and we may assume that $b^2=1$. Then $G=C_4\times \D_6$ and $G$ satisfies the conditions of (\ref{Mat2Q-1})(a). Otherwise, we may assume that $b^2=a^3$. Then $G=\GEN{a^4}\rtimes_4 \GEN{b}_8=C_3\rtimes_4 C_8$ and $G$ satisfies the conditions of (\ref{Mat2Q-1})(a).
Suppose that $i=-5$. Then $F=\Q(\zeta_3)$, $2\mid j$ and $(a^{-\frac{j}{2}}b)^2 = a^{3j}$.
By changing $b$ one may assume that $b^2=1$ or $b^2=a^6$. In the first case $G=C_3\times \D_8$ and in the second case $G=C_3\times \q_8$. Thus $G$ satisfies the conditions of (\ref{Mat2Q-3})(a).

%, we have $[H:K]=|H|=k$ and also $[N:H]=[\Q(\zeta_k):\Q]=2$, so $\varphi(k)=2$ and hence $k$ is either $3$, $4$ or $6$. Therefore the order of $G$ is either $6$, $8$ or $12$. Otherwise we want to find out the possible exceptional components of $\Q G$. In this way, by \cite[Proposition 3.4.]{ORS} $A\cong \Q(\zeta_k)\ast G/H$ with $G/H\cong C_2$. Thus $A=\Q(\zeta_k)[u| \zeta_k u=u\zeta_k^{\alpha}, u^2=\zeta_k^{\beta}]$ where $k=3,4$ or $6$. Calculating the values of $\alpha$ and $\beta$ in each case we get:
%
%\begin{itemize}
%
%\item [(\ref{Mat2Q})]  $A=\Q(\zeta_3)[u| \zeta_3 u=u\zeta_3^{2}, u^2=\zeta_3^{0}=1]\cong M_2(\Q)$
%
%\item [(\ref{Mat2Q-1})] $A=\Q(\zeta_4)[u| \zeta_4 u=u\zeta_4^{3}, u^2=\zeta_4^{0}=1]\cong M_2(\Q)$
%
%\item [] $A=\Q(\zeta_4)[u| \zeta_4 u=u\zeta_4^{3}, u^2=\zeta_4^{2}=-1]\cong \HQ(\Q)$
%
%\item [(\ref{Mat2Q-2})] $A=\Q(\zeta_6)[u| \zeta_6 u=u\zeta_6^{3}, u^2=\zeta_6^{0}=1]\cong M_2(\Q)$
%
%\item [] $A=\Q(\zeta_6)[u| \zeta_6 u=u\zeta_6^{5}, u^2=\zeta_6^{3}=-1]\cong (\frac{-3,-1}{\Q})$
%\end{itemize}
%
%

(C) Again in this case $\varphi(k)=4$ and hence $k=5,10,8$ or $12$. Moreover $N/H\cong \Gal(\Q(\zeta_k)/\Q)$.
If $k=5$ then $G=C_5\rtimes C_4$ and $\Q G \cong 2\Q \oplus \Q(i) \oplus M_4(\Q)$, which has not any exceptional component.
Thus $k\ne 5$.
If $k=10$ then $G=C_{10}\rtimes C_4=C_2\times (C_5\rtimes C_4)$ or $G=C_{10}\nsp_2 C_4=C_5\rtimes_2 C_8$.
However $\Q (C_2\times (C_5\rtimes C_4)) = 2 \Q(C_5\rtimes C_4)$ has not any exceptional component.
Then $G=C_5\rtimes_2 C_8$ and the unique exceptional component of $\Q G$ is $(\Q(\zeta_5)/\Q,-1)$. Thus $G$ satisfies the conditions of (\ref{Mat2-2-5Q}).

Assume that $k=8$. Then $G=\GEN{j,i,a}$ with $|j|=8$, $j^i=j\inv$, $j^a=j^3$, $i^2,a^2,(i,a)\in \GEN{j^4}$ and $(xy)^2\in \GEN{a^2}$.
If $a^2=j^4$ then $(aj)^2=1$ and hence we may assume that $a^2=1$. Suppose that $i^2=j^{4x}$ and $(ia)^2=a^{2y}$. Then $(i,a)=i^{-2}(ia)^2=j^{4x+2y}$. Thus $i^a=ij^{4x+2y}=j^{4x-2y}i$. Therefore $(j^{x-2y}i)^a=j^{x-2y}i$.
Thus replacing $i$ by $a^{x-2y}i$ we may assume that $(i,a)=1$.
If $x^2=1$ then $G \cong \D_{16}\rtimes C_2=\GEN{i,j}_{\D_{16}}\rtimes \GEN{a}$ and if $i^2=j^4$ then $G\cong \q_{16}\rtimes C_2=\GEN{i,j}_{\q_{16}}\rtimes \GEN{a}$. In both cases, $|j|=8$, $j^a=j^3$ and $(i,a)=1$.
Using Wedderga we obtain
    \begin{eqnarray*}
    \Q(\D_{16}\rtimes C_2) &=& 8\Q \oplus 2M_2(\Q) \oplus M_4(\Q). \\
    \Q(\q_{16}\rtimes C_2) &=& 8\Q \oplus 2M_2(\Q) \oplus M_2(\HQ(\Q)).
    \end{eqnarray*}
Therefore $G=\q_{16}\rtimes C_2$, $A=M_2(\HQ(\Q))$ and $G$ satisfies (\ref{Mat2HQQ})(a).

(D) In this case, $|K|$ divides $[N:K]=2k$. Thus $|G|$ divides $8k^2$.
Moreover $(\overline{H}=H/K,\overline{K}=1)$ is a strong Shoda pair of $\overline{N}=N/K$ satisfying the conditions of (A) (with $G, H$ and $K$ replaced by $\overline{N},\overline{H}$ and $\overline{K}$ respectively), except that now $\Q G e(\overline{N},\overline{H},\overline{K})$ is a totally definite algebra over $\Q$.
By the calculations in case (A), $\overline{N}=\q_{2k}$ with $k=4$ or $6$.
In the first case $A=M_2(\HQ(\Q))$, so $G$ satisfies the conditions of (\ref{Mat2HQQ})(b).
In the second case $A=M_2\quat{-1,-3}{\Q}$ and $G$ satisfies the conditions of (\ref{Mat2-1-3Q}).

%
%In this case $A\cong M_2(D)$ where $D\cong \Q(\zeta_k)\ast N/H$.  So $D$ is as follows:
%
%\begin{itemize}
%
%\item [(\ref{Mat2Q})]  $D=\Q(\zeta_3)[u| \zeta_3 u=u\zeta_3^{2}, u^2=\zeta_3^{0}=1]\cong M_2(\Q)$
%
%
%\item [(\ref{Mat2Q-1})] $D=\Q(\zeta_4)[u| \zeta_4 u=u\zeta_4^{3}, u^2=\zeta_4^{2}=-1]\cong \HQ(\Q)$
%
%\item [(\ref{Mat2Q-2})] $D=\Q(\zeta_6)[u| \zeta_6 u=u\zeta_6^{5}, u^2=\zeta_6^{3}=-1]\cong (\frac{-3,-1}{\Q})$
%\end{itemize}
%
%Thus $M_2((\frac{-3,-1}{\Q}))$ or $M_2(\HQ(\Q))$ is an exceptional component of $\Q G$.

(E) In this case (\ref{Phik}) reads $\varphi(k)=1$, and hence $k=2$. Moreover $|K|$ divides $[H:K]=k$ and hence $|K|=2$ and $|G|=8$.
Since $\q_8$ has not any exceptional component, $G\cong \D_8$. Therefore $G$ satisfies (\ref{Mat2Q}).

(F) Arguing as the previous case we deduce that $\varphi(k)=2$. Hence $k=3,4$ or $6$, $F=\Q(\zeta_k)$ and $|K|$ divides $k$.
Hence either $k=4$ and $A=M_2(\Q(\zeta_4))$ or $k=3$ or 6 $A\cong M_2(\Q(\zeta_3))$. In the first case $G$ satisfies (\ref{Mat2Q-1})(b) and in the second case it satisfies (\ref{Mat2Q-3})(b).
\end{proof}

The following proposition classify the metabelian CSP'-critical groups not included in Proposition~\ref{CSPDivision}.

\begin{proposition}\label{CSPMetabelian}
The following conditions are equivalent for a finite group $G$.
\begin{enumerate}
\item $G$ is CSP'-critical, it is not a subgroup of a division ring and $\Q G$ has an exceptional SSP component of type (EC2).
\item $G$ is metabelian and  CSP'-critical,  it is not a subgroup of a division ring and $\Q G$ has an exceptional component of type (EC2).
\item $G$ is one of the groups  in items (\ref{D6})-(\ref{G72}) of  Theorem~\ref{Main}.
\end{enumerate}
\end{proposition}

\begin{proof}
(3) implies (2). Let $G$ be one of the groups of types (\ref{D6})-(\ref{G72}) from Theorem~\ref{Main}.
Clearly $G$ is metabelian. By Proposition~\ref{Suficiencia}, $G$ is CSP'-critical.
Moreover, in all the cases $\Q G$ have an exceptional component of type (EC2) (see Table~\ref{InfoCSP}).
It remains to prove that $G$ is not a subgroup of a division ring. Otherwise,  one of the simple components of $\Q G$ is a division ring containing $G$ as an spanning set over $\Q$.
We have calculated the Wedderburn decomposition of $\Q G$ for all the groups $G$ in the proof of Proposition~\ref{Suficiencia}.
The only non-commmutative division algebra occurring in one of these Wedderburn decompositions is $\HQ(\Q)$ in the Wedderburn decomposition of $\Q(\q_8\times C_3)$. However $\q_8 \times C_3$ cannot be contained as an spanning set of $\HQ(\Q)$ over $\Q$ because the centre of $\HQ(\Q)$ has not elements of order 3.

(2) implies (1) is a direct consequence of the fact that every metabelian group is strongly monomial (Theorem~\ref{SSPmetabelian}).

(1) implies (3). We use the notation of Proposition~\ref{componentesdessp}.
Suppose that $G$ is CSP'-critical and $(H,K)$ is a strong Shoda pair of $G$ such that $A=\Q G e(G,H,K)$ is an exceptional component of type (EC2).
Then $\lambda_{H,K}^G$ is a faithful character of $G$, for $A$ is also a simple component of $\Q(G/\ker \lambda_{H,K}^G)$.
This implies that one of the conditions of Proposition~\ref{componentesdessp} hold.
In particular, the order of $G$ is bounded and hence the problem of deciding which groups satisfies (1) can be done in a finite number of computations.
Using this it is easy to write a program which calculate the list $\verb+L+$ of finite groups $G$ satisfying one of the conditions of Proposition~\ref{componentesdessp}. We have done this using GAP \cite{GAP} and Wedderga \cite{Wedderga} (see Section~\ref{Program}).
The list $\verb+L+$ contains all the groups satisfying (1) but it contains some groups not satisfying (1).
For example, if $G\in \verb+L+$ and $G/N$ satisfies (3) for some non-trivial normal subgroup $N$ of $G$ then $G$ is not CSP'-critical.
So we calculate, using GAP, a second list \verb+R+ excluding from $\verb+L+$ all the elements having a proper quotient  satisfying (3).
The groups forming \verb+R+ are precisely the groups of types (\ref{D6})-(\ref{G72}) from Theorem~\ref{Main}. This finishes the proof.

%\begin{center}
%\begin{tabular}{|l|l|}
%\hline
%$M_2(\Q)$ & [ 2, Rationals ], \\
%    & $(\Q(\zeta_3)/\Q,1)=$[ 1, Rationals, 3, [ 2, 2, 0 ] ], \\
%    & $(\Q(\zeta_6)/\Q,1)=$[ 1, Rationals, 6, [ 2, 5, 0 ] ]  \\\hline
%$M_2(\Q(\zeta_3))$ & [ 2, CF(3) ], \\
%    &$(\Q(\zeta_{12}/\Q(\zeta_3),-1)=\quat{-1,-3}{\Q(\zeta_3)}=$[ 1, CF(3), 4, [ 2, 3, 2 ] ], \\
%    &[ 1, CF(3), 12, [ 2, 7, 6 ] ] \\\hline
%$M_2(\Q(\zeta_4))$ & [ 2, GaussianRationals ], \\
%    & [ 1, GaussianRationals, 12, [ 2, 5, 0 ] ], [ 1, GaussianRationals, 12, [ 2, 5, 6 ] ] \\\hline
%$M_2(\HQ(\Q))$ & [ 2, Rationals, 4, [ 2, 3, 2 ] ] \\\hline
%$\quat{-3,i}{\Q(i)}$ & [ 1, GaussianRationals, 12, [ 2, 5, 9 ] ], [ 1, GaussianRationals, 12, [ 2, 5, 3 ] ]\\ \hline
%\end{tabular}
%\end{center}
\end{proof}

%\vspace{ 8pt}
%
%\begin{tabular}{|l | l |}
%\hline
% $\D_6$ & [6,1]\\
%\hline
%$\D_8$ & [8,3]\\
%\hline
%$\GEN{a,b | a^8=1, b^2=1, a^b=a^5}$& [16,6]\\
%\hline
%$\GEN{a,b,c | a^4=1, b^2=1, c^2=a^2, a^b=a^{-1}, a^c=a^{-1}, (b,c)=1}$&[16,13]\\
%\hline
%$\q_8\times C_3$&[24,11]\\
%\hline
%$\q_8 \Ydown_2 \D_8$&[32,50]\\
%\hline
%$C_5\rtimes C_8$&[40,3]\\
%\hline
%$\GEN{a, b, c | a^3=1, b^3=1, c^8=1, a^c=b^{-1}, b^c=a, (a,b)=1}$&[72,19]\\
%\hline
%\end{tabular}
%
%\vspace{8pt}

%\ChMau{VAMOS CON LOS GRUPOS IMPRIMITIVOS}

Now we deal with the case where $G$ is an imprimitive subgroup of an exceptional component of type (EC2).
For that we need the following lemma from \cite{baniq} and three additional lemmas.

\begin{lemma}\label{imprimitivos}\cite[Lemma~2.2]{baniq}
Let $D$ be a division algebra and let $G$ be a subgroup of $\GL_2(D)$ spanning $M_2(D)$ over $\Q$.
Then $G$ is imprimitive if and only if there is a group homomorphism $\theta:N\rightarrow \GL_1(D)$ for $N$ a subgroup of index 2 of $G$ and $g\in G\setminus N$ such that if $K=\ker \theta$ then $K\cap K^g=1$.
\end{lemma}

Observe that the hypothesis that $G$ spans $M_2(D)$ is not present in the statement of \cite[Lemma~2.2]{baniq} while it is used in its proof. This hypothesis is necessary because otherwise $G$ could be included diagonally on $M_2(D)$. For example, if $G$ is a subgroup of a division ring then one can embedded $G$ diagonally on $\GL_2(D)$. This is an imprimitive representation and if the order of $G$ is odd then the existence of the subgroup $N$ mentioned in the lemma is not possible.

%\ChAngel{\sout{Let $\theta:N\rightarrow \GL_1(D)$, $g\in G$ and $K=\ker \theta$ satisfy the conditions of Lemma~\ref{imprimitivos}.
%Then}
%$$\cancel{n\in N \rightarrow \begin{pmatrix}  \theta(n) & 0 \\ 0 & \theta(g^{-1} n g) \end{pmatrix} , \; g\rightarrow \begin{pmatrix}  0 & \theta(g^2)  \\ 1 & 0 \end{pmatrix}}$$
%\sout{defines an injective group homomorphism $\overline{\theta}:G\rightarrow \GL_2(D)$.
%Moreover, $D$ is generated over $\Q$ by $\theta(N)$, since $M_2(D)$ is spanned over $\Q$ by $\overline{\theta}(G)$.
%Thus $D$ is spanned over $\Q$ by a subgroup isomorphic to $N/K$ and hence $D$ is a simple component of $\Q(N/K)$.}}

\begin{lemma}\label{ImprimitiveSSP}
Assume that $G$ is an imprimitive subgroup of an exceptional component $M_2(D)$ of type (EC2) of $\Q G$ and let $\theta:N\rightarrow \GL_1(D)$, $g\in G$ and $K=\ker \theta$ satisfy the conditions of Lemma~\ref{imprimitivos}.
If $D$ is an SSP component of $\Q(N/K)$ then $M_2(D)$ is an SSP component of $\Q G$.
\end{lemma}

\begin{proof}
The rules
$$n\in N \rightarrow \begin{pmatrix}  \theta(n) & 0 \\ 0 & \theta(g^{-1} n g) \end{pmatrix} , \; g\rightarrow \begin{pmatrix}  0 & \theta(g^2)  \\ 1 & 0 \end{pmatrix}$$
define an injective group homomorphism $\overline{\theta}:G\rightarrow \GL_2(D)$.
Moreover, $D$ is generated over $\Q$ by $\theta(N)$, since $M_2(D)$ is spanned over $\Q$ by $\overline{\theta}(G)$.
Thus $D$ is spanned over $\Q$ by a subgroup isomorphic to $N/K$ and hence $D$ is a simple component of $\Q(N/K)$.
Suppose that $(H/K,K_1/K)$ is a strong Shoda pair of  $N/K$ such that $D=\Q (N/K) e(N/K,H/K,K_1/K)$.
Then $\theta$ is an irreducible representation of $N$ affording the character $\lambda_{H,K_1}^N$. As $K=\ker \theta$, $\theta$ lifts to a faithful representation $\rho$ of $N/K$ affording the character $\lambda_{H/K,K_1/K}^{N/K}$.
Using (\ref{SSPComponent}) and the fact that $D$ is a division algebra, we deduce that $K_1/K$ is normal in  $N/K$.
Thus $K_1/K=\core_{N/K}(K_1/K)=\ker \lambda_{H/K,K_1/K}^{N/K}=1$ so that $K_1=K$.
As $K\unlhd N$ and $H/K$ is cyclic and maximal abelian and normal in $N/K$, we deduce from \cite[Corollary 3.6]{ORS} that $(H,K)$ is a strong Shoda pair of $N$ and $e(N,H,K)=\varepsilon(H,K)$ is a primitive central idempotent of $\Q N$.
Since $N\unlhd G$, the $G$-conjugates of $\varepsilon(H,K)$ are primitive central idempotents of $\Q N$ and hence they are orthogonal.
If $N_G(K)=N$ then $(H,K)$ is a strong Shoda pair of $G$. Then $\Q e(G,H,K)\cong M_2(D)$, by (\ref{SSPComponent}).
Otherwise, $K\unlhd G$ and hence $K=K\cap K^g=1$. If $H$ is maximal abelian in $G$ then again $(H,K)$ is strong Shoda pair of $G$. Moreover $\overline{\theta}$ affords the character $\lambda_{H,1}^G$ and hence $M_2(D)=\Q G e(G,H,K)$.
Assume otherwise that $H$ is not maximal abelian in $G$. Then there is $n\in N$ such that $n\inv g\in \Cen_G(H)$.
As $(H,1)$ is a strong Shoda pair of $N$, $H$ is normal in $N$ and therefore  $h^g=h^n \in H$ for every  $h\in H$.
Therefore $H\unlhd G$. Moreover $[N:H]$ is the degree of $D$ which is either 1 or 2.
Thus $G/H$ is abelian. Since $H$ is cyclic, we deduce that $G$ is metabelian and hence every simple component of $\Q G$ is SSP.
So in all the cases $M_2(D)$ is an SSP component of $\Q G$.
\end{proof}

\begin{lemma}\label{losHs}
Let $N$ be a finite group containing two normal subgroups $K_1$ and $K_2$ such that $K_1\cong K_2$, $K_1\cap K_2=1$ and $N/K_i\cong \SL(2,3).$ Then $N$ is isomorphic to one of the following groups:
\begin{enumerate}
\item $\SL(2,3)$.
\item $\SL(2,3)\times \SL(2,3).$
\item $(\q_8\times \q_8)\rtimes C_3=(\GEN{i_1,j_1}_{\q_8}\times \GEN{i_2,j_2}_{\q_8})\rtimes \GEN{c}$, with $i_k^c=j_k$ and $j_k^c=i_kj_k$ for every $k=1,2$.
\item $\SL(2,3)\times C_2.$
\end{enumerate}
\end{lemma}

\begin{proof}
Clearly if $K_1=1$ then $N\cong \SL(2,3)$. So we assume that $K_1\ne 1$.
The assumption of the lemma implies that $K_1\times K_2$ is a normal subgroup of $N$.
Then $K_1$ is isomorphic to a non-trivial normal subgroup of $N/K_2\cong \SL(2,3)$.
Thus $K_1$ is isomorphic to either $C_2$, $Q_8$ or $\SL(2,3)$.
%Thus  $|N|$ is either $48$, $192$ or $576$.
We discuss each case separately.

If $K_1\cong \SL(2,3)$ then $|N|=576=|K_1\times K_2|$. Thus $N=K_1\times K_2\cong \SL(2,3)\times \SL(2,3)$.

Assume $K_1 \cong \q_8$. Then $|N|=192$ and $[N:K_1\times K_2]=3$.
Hence $N/(K_1\times K_2)\cong \GEN{\overline{c}}_3$ for some $c\in N$. As $N/K_k\cong \SL(2,3)$, one can choose generators $i_k$ and $j_k$ of $K_k$ such that the action of $c$ on $K_k=\GEN{i_k, j_k}$ is given by $i_k^c=j_k$ and $j_k^c=i_kj_k$. Therefore $N\cong (\q_8\times \q_8)\rtimes C_3$ and the conditions of (3) hold.

Finally suppose that $K_1=\GEN{z_i}_2$. Then $|N|=48$ and $K_i\leq Z(N)$. Let $P$ be a Sylow $2$-subgroup of $N$ and $c$ a generator of a Sylow $3$-subgroup of $N$.
Since $N/K_i\cong \SL(2,3)=\q_8\rtimes C_3$, we have $\q_8\cong P/K_i \unlhd N/K_i$.
Thus $P\unlhd N$ $P$ is non-abelian of order 16 and $K_1\times K_2 \subseteq Z(P) \subset P$.
As  $P/Z(P)$ is not cyclic, necessarily $Z(P)=K_1\times K_2\cong C_2^2\cong P/Z(P)$.
Let $a,b\in P$ such that $\langle aK_1, bK_1 \rangle = P/K_1$.
Then $P=\langle a, b, z_1 \rangle$ where $a$ and $b$ have order $4$ and $a^2, b^2, (a,b)\in (K_1\times K_2)\setminus (K_1\cup K_2) = \{ z_1z_2\}$.
Thus $a^2=b^2=(a,b)=z_1z_2$ and this implies that $\GEN{a,b}\cong \q_8$ and $\langle a, b\rangle\cap  \langle z_1\rangle=1$.
Hence $P=\GEN{a,b}\times K_i \cong \q_8\times C_2$.
Now it is easy to check that $P$ has exactly four subgroups isomorphic to $\q_8$.
Namely $\langle a, b\rangle, \langle az_1, b\rangle, \langle a, bz_1\rangle$ and  $\langle az_1,bz_1\rangle$.
The action of $\GEN{c}_3$ on $P$ permutes these groups and hence it leaves invariant one of them.
So let $\q_8\cong Q\le P$ with $Q^c=Q$.
Observe $Q\cap K_1=Q\cap K_2=1$, so after a suitable change of generators we may assume that $Q=\GEN{a,b}$. Since $z_2^c=1$, we necessarily have that the action of $c$ on $Q$ has order $3$ and hence $\GEN{Q,c}=\SL(2,3)$ and $N=\GEN{Q,c}\times K_1\cong \SL(2,3)\times C_2$.
\end{proof}

\begin{lemma}\label{G1152}
Let $G$ be a group and let $N=K_1\times K_2$ be a subgroup of index 2 of $G$ with $K_1\cong K_2\cong \SL(2,3)$.
If the action of $G$ on $N$ permutes $K_1$ and $K_2$ then $G=(\SL(2,3)\times \SL(2,3))\rtimes \GEN{g}_2$ with $(x,y)^g=(y,x)$, for every $(x,y) \in \SL(2,3)\times \SL(2,3)$.
\end{lemma}

\begin{proof}
We may assume that $N=\SL(2,3)\times \SL(2,3)$, $K_1=\SL(2,3)\times 1$ and $K_2=1\times \SL(2,3)$.
For every $s\in \SL(2,3)$, let $c_s$ denote the inner automorphism of $\SL(2,3)$ given by $c_s(x)=x^s$.

Let $g_1\in G\setminus N$ and suppose $g_{1}^2=(s_1,s_2)$. By assumption, there are $\tau,\sigma\in \Aut(\SL(2,3))$ such that $(x,y)^{g_1}=(\tau(y),\sigma(x))$ for every $(x,y)\in N$. Then 
    $$(c_{s_1}(x),c_{s_2}(y))=(x,y)^{g_{1}^2}=(\tau\sigma(x),\sigma\tau(y)),$$ or equivalently $c_{s_1}=\tau\sigma$ and $c_{s_2}=\sigma\tau$.
Therefore $(x,y)^{g_1}=(\sigma^{-1}(y)^{s_1},\sigma(x))=(\sigma^{-1}(y^{s_2}),\sigma(x))$ for every $x,y\in \SL(2,3)$.

We claim that $s_2=\sigma(s_1)$. Indeed, if $x\in \SL(2,3)$ then $\tau(x)=\sigma^{-1}c_{s_2}(x)=\sigma^{-1}(s_2^{-1}xs_2)=c_{\sigma^{-1}(s_2)}\sigma^{-1}(x)$.
Therefore $c_{\sigma^{-1}(s_2)}=\tau\sigma=c_{s_1}$ and hence if $z=s_1\inv \sigma\inv(s_2)$ then $z\in \SL(2,3)$ and  $s_2=\sigma(s_1z)$.
Consequently $(s_1,\sigma(s_1z))=g_1^2=(g_1^2)^{g_1}=(c_{s_1} (s_1z),\sigma(s_1))=(s_1z,\sigma(s_1))$.
Thus $z=1$ and the claim follows.

Let $g_2=g_1 (s_1\inv,1)$. Then $$g_2^2= g_1^2 (s_1\inv,1)^{g_1} (s_1,1) = (s_1,\sigma(s_1)) (1,\sigma(s_1\inv))(s_1\inv,1)=1.$$
Therefore, one may assume without loss of generality that $g_1^2=1$.
Hence, by the previous paragraphs $(x,y)^{g_1}=(\sigma\inv(y),\sigma(x))$.
Let $K=(\SL(2,3)\times \SL(2,3))\rtimes \GEN{g}_2$, with $(x,y)^g=(y,x)$. Then the map $\alpha:(x,y)g^i\in K\mapsto (x,\sigma(y))g_1^i\in G$ is an isomorphism because $\alpha(g(x,y))=\alpha((y,x)g)=(y,\sigma(x))g_1=g_1(x,\sigma(y))=\alpha(g)\alpha(x,y)$.
\end{proof}

We are ready to prove Theorem~\ref{Main} for non-metabelian imprimitive subgroups of two-by-two matrix rings over division rings.

\begin{proposition}\label{CSPImprimitive}
Let $G$ be a non metabelian CSP'-group which is an imprimitive subgroup of an exceptional component of type (EC2) of $\Q G$. Then $G$ is one of the groups  in items (\ref{G96})-(\ref{G384})  of Theorem~\ref{Main}.
\end{proposition}

\begin{proof}
Let $G$ satisfy the hypothesis of the proposition. and let  $D$ be a division ring such that $M_2(D)$ is an exceptional component of $\Q G$ and $G$ is an imprimitive subgroup of $M_2(D)$.
Let $\theta:N\rightarrow \GL_1(D)$, $K$ and $g\in G$ be as in Lemma~\ref{imprimitivos}.
By Lemma~\ref{ImprimitiveSSP}, $N/K$ is a subgroup of $\GL_1(D)$ and $D$ is a simple component of $\Q(N/K)$.
We claim that $G$ is not a subgroup of a division ring. Otherwise, by Proposition~\ref{CSPDivision}, $G$ is one of the groups of types (\ref{Zgrupos1})-(\ref{SL5}) of Theorem~\ref{Main}. However, the first three types are metabelian while the other two groups have not a subgroup of index 2. This proves the claim.
Since $G$ is not metabelian, by Proposition~\ref{CSPMetabelian}, the exceptional component $M_2(D)$ of $\Q G$ is not SSP.
Then, by Lemma~\ref{ImprimitiveSSP}, $D$ is not an SSP component of $\Q(N/K)$. Therefore $N/K$ is not strongly monomial.
In particular, $N/K$ is non-abelian. Thus $D$ is not comutative and hence it is a totally definite quaternion algebra over $\Q$.
Moreover, $N/K$  is not a Z-group and therefore it is one of the groups in item (NZ) of Theorem~\ref{SubgruposAD}. However the groups in (NZ)(b) and (NZ)(c) are metabelian and, in particular, strongly monomial, while the only non-commutative division algebra in the Wedderburn decomposition of $\Q \bo^*$ is $\HQ(\Q(\sqrt{2}))$ and the only division algebra in the Wedderburn decomposition of $\Q \SL(2,5)$ is $(\Q(\zeta_5)/\Q(\sqrt{5}),-1)$.  Hence $N/K=\SL(2,3)\times M$ with $M$ a Z-group of odd order. If $M$ is non-abelian then every Wedderburn component containing  $N/K$ should have degree greater than 2.
Hence $M$ is cyclic.
If $M\ne 1$ then the centre of $D$ should contain a root of unit of order greater than $2$, in contradiction with the fact that $Z(D)=\Q$. Thus $N/K\cong \SL(2,3)$.
This implies that $D=\HQ(\Q)$ and hence the order of the centre of $G$ is at most 2.
Moreover, $K\cap K^g=1$ and hence the subgroups $K_1=K$ and $K_2=K^g$ satisfy the hypothesis of Lemma~\ref{losHs}. Thus one may assume that one of the following conditions holds:
(1) $N= \SL(2,3)\times \SL(2,3)$ and $K=\SL(2,3)\times 1$ and $K^g=1\times \SL(2,3)$;
(2) $N= \SL(2,3)$ and $K=1$;
(3) $N= \SL(2,3)\times \GEN{z}_2$ and $K=\GEN{z}_2\ne K^g$;
(4) $N= (\q_8\times \q_8)\rtimes C_3$, $K=\q_8\times 1$ and $K^g=1\times \q_8$.

(1) Assume first that $N=\SL(2,3)\times \SL(2,3)$. Then, by Lemma~\ref{G1152},  $G\cong N\rtimes \GEN{g}_2$ with $(x,y)^g=(y,x)$.
Let $H=\GEN{P_2,b_1b_2}$, where $P_2$ is the Sylow 2-subgroup of $N$ and $\GEN{b_1,b_2}\cong C_3^2$ is a Sylow 3-subgroup of $N$.
Then $P_2^g=P_2$ and $(b_1b_2)^g=b_1b_2$. Then $H$ is a normal subgroup of $G$ and $G/H\cong \D_6$ in contradiction with (P1).

(2) Suppose now that $N\cong \SL(2,3)=\GEN{i,j}_{\q_8}\rtimes \GEN{a}_3$ and let $g\in G\setminus N$ and $Q=\GEN{i,j}$. Then $Q$ is a characteristic subgroup of $\SL(2,3)$ and hence $Q\lhd G$.
One may assume without loss of generality that $g^2\in Q$.
Then $N$ has eight elements of order 3 in two conjugacy classes $a^N=\{a,ia,ja,ija\}$ and $(a^2)^N = \{a^2,i^{-1}a^2,j^{-1}a^2,i^{-1}ja^2\}$.
If $a^g \in (a^2)^N$ then $G/N_2\cong \D_6$, in contradiction with (P1).
Therefore $a^g=a^q$ for some $q\in N$ and, as $Q$ contains a transversal of $C_N(a)$, we may assume that $q\in Q$.
Replacing $g$ by $gq\inv$, one may assume that $(a,g)=1$.
Let $\alpha$ be the restriction to $Q$ of the inner automorphism defined by $ag$.
Since the order of $a$ is 3 and the order of $q$ is a power of 2, the order of $\alpha$ is multiple of $3$.
Moreover $\Aut(Q)\cong S_4$ and therefore the order of $\alpha$ is $3$.
Then $(g,Q)=1$ and hence $g\in Z(G)$.
Thus $g^2=1$ and $G/\GEN{g}\cong \SL(2,3)$ in contradiction with (P1).

(3) Suppose now that $N=\SL(2,3)\times C_2=(\GEN{i,j}_{\q_8}\rtimes \GEN{a}_3)\times \GEN{z}_2$ and $K=\GEN{z}_2$.
The Sylow 2-subgroup of $N$ is $N_2=\GEN{i,j,z}\cong \q_8\times C_2$.
Hence $N_2$ is characteristic in $N$ it has three elements of order 2, namely $i^2,z$ and $i^2z$, and $i^2$ is the only one which is a square.
Thus $z^g=i^2z$ because $K\ne K^g$.
As in the previous case $N$ has eight elements of order 3 in two conjugacy classes $a^N=\{a,ia,ja,ija\}$ and $(a^2)^N = \{a^2,i^{-1}a^2,j^{-1}a^2,i^{-1}ja^2\}$ and if $a^g \in (a^2)^N$ then $G/N_2\cong \D_6$, in contradiction with (P1).
Therefore $(a,g)=1$ and thus $g^2\in \GEN{i^2}$. If $g^2=1$ then $(gz)^2=i^2$.
Hence we may assume that $g^2=i^2$. Then $g^z=zgz=gi^2=g^3$ and hence $\GEN{g,z}\cong \D_8$.
Moreover $N_2$ has 4 subgroups isomorphic to $\q_8$: $\GEN{i,j}, \GEN{i,jz}. \GEN{iz,j}$ and $\GEN{iz,jz}$.

We claim that $(g,i)=(g,j)=1$. If $\GEN{i,j}^g=\GEN{i,j}$ then conjugation by $ag$ induces an automorphism of $\GEN{i,j}\cong \q_8$ of order multiple of 3.
As $\Aut(\q_8)\cong S_4$ this implies that $g$ commutes with $i$ and $j$ and the claim follows. Otherwise, i.e. if $\GEN{i,j}^g\ne \GEN{i,j}$, the intersection of $\GEN{i,j}$ and $\GEN{i,j}^g$ is a cyclic subgroup of order $4$ and, by symmetry, one may assume that $\GEN{i,j}\cap \GEN{i,j}^g=\GEN{i}$.
Then $\GEN{i,j}^g=\GEN{i,jz}$.
If $j^g\in \GEN{i}$ then $\GEN{i^g}=\GEN{j^{g^2}}=\GEN{j}$  and hence $\GEN{i,j}^g=\GEN{i,j}$, contradicting the assumption.
Thus $j^g\in \{jz,j\inv z,ijz,i\inv j z\}$ and $j^{g^2}=j^{i^2}=j$.
If $j^g=j z$ or $j^g=j\inv z$ then $j=j^{g^2}=(jz)^g = j z i^2 z = i^2j$ or $j=j^{g^2}=(j\inv z)^g = j z i^2 z = i^2 j$, a contradiction. Thus $j^g = ijz$ or $j^g=i\inv j z$ and, replacing $i$ by $i\inv$ if necessary one may assume that $j^g=ijz$. Then $j=j^{g^2}=(ijz)^g = i^g ijz i^2 z = i^g i\inv j$ and therefore $i^g=i$. Thus $i^{ag}=j^g=ijz\ne j = i^a = i^{ga}$, in contradiction with $ga=ag$. This finishes the proof of the claim.

We conclude that $(\GEN{g,z},\SL(2,3))=1$, $\GEN{g,z}\cong \D_8$ and $\GEN{g,z}\cap \SL(2,3)=\GEN{g^2}=\GEN{i^2}$. Thus $G=\SL(2,3)\Ydown_2 \D_8$  i.e. $G$ is the group of item~(\ref{G96}) in Theorem~\ref{Main}.

(4) Finally suppose that $N\cong (\q_8\times \q_8)\rtimes C_3=(\GEN{i_1,j_1}_{\q_8}\times \GEN{i_2,j_2}_{\q_8})\rtimes \GEN{c}$, with $i_k^c=j_k$, $j_k^c=i_kj_k$ and $\GEN{i_1,j_1}^g=\GEN{i_2,j_2}$. As in the previous case $N$ has two conjugacy classes formed by elements of order $3$ represented by $c$ and $c^2$. If $c$ and $c^2$ are conjugate in $G$ and $N_2=\GEN{i_1,j_1,i_2,j_2}$ then $G/N_2\cong \D_6$,   contradicting (P1). As in the previous case this implies that we may assume that $(c,g)=1$ and, in particular, $g^2\in Z(N)=\GEN{i_1^2,i_2^2}$. However $((i_1)^2)^g=i_2^2$ and therefore $g^2\not\in \{i_1^2,i_2^2\}$.
If $g^2=i_1^2i_2^2$ then $(i_1^2g)^2=1$, so we may assume also that $g^2=1$.
On the other hand $N$ has eight normal subgroups isomorphic to $\q_8$, namely
    $$\matriz{{cccc}
    Q_{11}=\GEN{i_1,j_1}, & Q_{12}=\GEN{i_1i_2^2,j_1}, & Q_{13}=\GEN{i_1,j_1i_1^2}, & Q_{14}=\GEN{i_1i_2^2,j_1i_2^2},\\
    Q_{21}=\GEN{i_2,j_2}, & Q_{22}=\GEN{i_1^2i_2,j_2}, & Q_{23}=\GEN{i_2,i_1^2i_2}, & Q_{24}=\GEN{i_1^2i_2,i_1^2i_2}.}$$
Observe that $Q_{1x}\cap Q_{2y}=1$ and $(Q_{1x},Q_{2y})=1$ for every $x,y\in \{1,2\}$.
Moreover, $N$ has three elements of order $2$, namely $i_1^2$, $i_2^2$ and $i_1^2i_2^2$.
Furthermore $i_x^2\in Q_{1x}$ and $i_x^2\in Q_{2x}$ for every $x=1,2,3,4$. As $i_1^g=i_2$, we deduce that the action of $g$ by conjugation interchange the $Q_{1x}$'s with the $Q_{2x}$'s.
Therefore, if $b=c^2g$ then $b$ has order $6$ and the action of $b$ by conjugation interchange the $Q_{1x}$'s with the $Q_{2x}$'s.
Since the action of $c$ permutes transitively the three cyclic subgroups of order 4 of each $Q_{ix}$, after renaming the generators we may assume that
$i_2=i_1^b$ and $j_2=j_1^b$.
As $c=b^2$, we have $i_2^b=j_1$ and $j_2^b=i_1j_1$ and
$G=(\GEN{i_1,j_1}\times \GEN{i_2,j_2})\rtimes \GEN{b}_6$. Thus $G$ is the group of item (\ref{G384}) in Theorem~\ref{Main}.
\end{proof}

The following proposition completes the proof of Theorem~\ref{Main}.

\begin{proposition}\label{CSPPrimitive}
Let $G$ be a CSP'-critical group which does not satisfy the hypothesis of neither Proposition~\ref{CSPDivision}, nor Proposition~\ref{CSPMetabelian} nor Proposition~\ref{CSPImprimitive}. Then $G$ is one of the groups in items (\ref{G48})-(\ref{G1920})  of Theorem~\ref{Main}.
\end{proposition}

\begin{proof}
By the hypothesis, $G$ satisfies (P1) and (P2). Moreover we may assume that $G$ is a primitive subgroup of an exceptional component $M_2(D)$ of type (EC2) of $\Q G$.
In particular $D$ is either $\Q$, an imaginary quadratic extension of $\Q$ or a totally definite quaternion algebra over $\Q$.
Moreover, $G$ satisfies the following properties:
\begin{itemize}
\item[(P3)] \label{NoSMon} $G$ is not strongly monomial. In particular, $G$ is not abelian-by-supersolvable.
\item[(P4)] \label{OrdenAcotado} If $n$ is the order of an element $g$ of $G$ then $\varphi(n)\le 4$ and if $g\in Z(G)$ then $\varphi(n)\le 2$.
\end{itemize}
(P3)  is a consequence of Proposition~\ref{CSPMetabelian}. To prove (P4) suppose $F=Z(D)$ and let $g\in G$.
Then $F$ is either $\Q$ or an imaginary quadratic extension of $\Q$.
If $g\in Z(G)$ then $F$ has a root of unity of order $n$ and then $\varphi(n)\le [F:\Q]\le 2$.
Assume that $D=F$. Consider $g$ as an element of $M_2(F)$ and let $f$ be the characteristic polynomial of $g$.
Then $g$ is conjugate in $M_2(\C)$ to a diagonal matrix $\diag(\alpha_1,\alpha_2)$ where $\alpha_1$ and $\alpha_2$ are roots of unity.
Let $m_1$ and $m_2$ be the orders of $\alpha_1$ and $\alpha_2$. Then $n=\lcm(m_1,m_2)$.
Furthermore $f(\alpha_1)=f(\alpha_2)=0$ and hence $F(\zeta_n)=F(\alpha_1,\alpha_2)$ is contained in the splitting field $K$ of $f$ over $F$.
Thus $[F(\zeta_n):F]\le 2$ and hence $\varphi(n)=[\Q(\zeta_n):\Q]\le [F(\zeta_n):\Q] \le [F(\zeta_n):F][F:\Q]\le 4$.
Suppose otherwise that $D\ne F$. Then $D$ is a totally definite quaternion algebra over $\Q$ and $M_2(D)$ is embedded in $M_4(\C)$.
Consider $g$ as an element of $M_4(\C)$ and let $f$ be the characteristic polynomial of $g$.
Then $g$ is conjugate in $M_4(\C)$ of a diagonal matrix $\diag(\alpha_1,\alpha_2,\alpha_3,\alpha_4)$ where each $\alpha_i$ a root of unity, say of order $m_i$, and $n=\lcm(m_1,m_2,m_3,m_4)$.
Then $f$ is multiple of the least common multiple of the minimal polynomials of the $\alpha_i$ and $\Q(\zeta_n)=\Q(\alpha_1,\alpha_2,\alpha_3,\alpha_4)$.
As $f$ has degree 4, either $\Q(\zeta_n)=\Q(\alpha_i)$ for some $i$ or $\Q(\zeta_n)=\Q(\alpha_i,\alpha_j)$, with $\Q(\alpha_i)\ne \Q(\alpha_j)$  and $\varphi(m_1)=\varphi(m_2)=2$.
In the first case $\varphi(n)=\varphi(m_i)\le 4$ and in the second case $\varphi(n)=\varphi(m_i)\varphi(m_2)=4$.

In the remainder of the proof our main tool is  the classification of the primitive subgroups of two-by-two matrix rings over division rings obtained by Banieqbal in \cite{baniq}.
This classification is contained in Theorems~3.8, 4.4, 4.5, 4.6, 4.7 and 5.8 of \cite{baniq}.
So $G$ is one of the groups appearing in these theorems and we will consider each case separately.

We start observing that the groups of \cite[Theorem~3.8]{baniq} are all supersolvable and hence they do not satisfy (P3).

In the description of many of the remaining groups it has a relevant role the group $G_{m,r}=\GEN{a}_m\nsp_s \GEN{b}_n$ from (\ref{Gmr}) with $m,r,s,t$ and $n$ satisfying the conditions of (\ref{GmrCond}).

Let $G$ be one of the groups of \cite[Theorem~4.4]{baniq} and let $N=O_2(G)\Cen_G(O_2(G))$.
The statement of the theorem considers  eight types for $G$ denoted (a$_1$), (a$_2$), (b$_1$), (b$_2$), (c), (d$_1$), (d$_2$), and (e).
An inspection of the proof of \cite[Theorem~4.4]{baniq} shows that if $G$ satisfies satisfies conditions (b$_1$), (b$_2$), (c) or (e) then $G/N\cong \D_6$ in contradiction with (P1).
If $G$ satisfies (a$_1$) or (a$_2$) then $G$ contains $T_{\alpha}^*$ with $\alpha\ge 1$.
Therefore $G$ has an element of order $3^{\alpha}$ and hence (P4) implies that $\alpha=1$. Since $T_1^*\cong \SL(2,3)$, in case (a$_1$), $G=\SL(2,3)\times N$ for some $N$.
Then  (P1) implies that $G=\SL(2,3)$, a subgroup of a division ring.
This excludes this case. If $G$ is of type (a$_2$) then $G=\SL(2,3)\Ydown_2 G_{m,r}$, with $v_2(s)=v_2(n)=1$. Then $1\ne n = o_m(r)$ and therefore $m> 2$. Then $a^2\ne 1$ and $G/\GEN{a^2,b^4}\cong \SL(2,3)\Ydown_2 C_4$ in contradiction with (P1).
If $G$ is of type (d$_1$) or (d$_2$) then $G\cong \q_8\nsp G_{m,r}$ with $m$ odd,  $3\mid n$, $(\q_8,a)=1$, $i^b=j$ and $j^b=ij$.
Then $G/\GEN{a,b^3}\cong \SL(2,3)$ in contradiction with (P1).

Suppose now that $G$ is one of the groups of \cite[Theorem~4.5]{baniq} and let again $N=O_2(G)\Cen_G(O_2(G))$.
In this case Banieqbal consider 13 cases denoted (a), (b), $\dots$, (h), (i$_1$), (i$_2$), (j), \dots, (l).
As in the previous paragraph for some of these cases, namely all but the first three, $G/N\cong \D_6$ in contradiction with (P1).
Hence $G$ satisfies either (a), (b) or (c).
In case (a), $G=T^*_{\beta}\Ydown_2 G_{m,r}$ with $2\le v_2(s)$. By (P4), $\beta=1$ and hence $G=\SL(2,3)\Ydown_2 G_{m,r}$ with $4\mid s$ and $t$ odd.
Moreover, by (P4), $4\ge \varphi(m)\ge 2\varphi(t)$. Hence $\varphi(t)\le 2$ and, as $t$ is odd,  necessarily $t=1$ or $3$.
If $t=3$ then $\gcd(r-1,m)=s=4$ and hence $m=12$. This implies that $r\equiv 5 \mod 12$ and we may assume that $r=5$ and $n=o_{12}(5)=2$.
Then $G/\GEN{\SL(2,3),a^3}\cong \D_6$ in contradiction with (P1). Thus $t=1$ and hence $G_{m,r}$ is cyclic of order $m$ with $4\mid m$. Then $\varphi(m)=2$, by (P4), and therefore $m=4$. We conclude that $G=\SL(2,3)\Ydown_2 C_4$, i.e. $G$ is as in item (\ref{G48}) of Theorem~\ref{Main}.
In case (b), $G=\SL(2,3)\Ydown_2 \D_{2^{\alpha+1}m}$ with $\alpha\ge 2$ and $m$ is odd and greater than $1$. Let $a$ be an element of $\D_{2^{\alpha+1}m}$ of order $2^{\alpha}m$. Then $a^4\ne 1$ and $G/\GEN{a^4}\cong \SL(2,3)\Ydown_2 \D_8$, a contradiction with (P1). Finally, in case (c), $G=\q_8 \nsp_2 G_{m,r}$ with $4\nmid s$ and $3\mid n=o_m(r) \mid \varphi(m)\le 4$. This implies that $3=\varphi(m)$ which is not possible because $3$ is not in the image of the Euler function.

Assume that $G$ is one of the groups of \cite[Theorem~4.6]{baniq} and take now $N=O_2(G)$.
Now there are three types (a), (b) and (c) to consider.
In case (a), $G=\SL(2,3)\Ydown_2 \D_{2^{\alpha}}$ with $\alpha\ge 4$.
Then a proper quotient of $G$ is isomorphic to $\D_8$, in contradiction with (P1).
Moreover, from the proof of the theorem one has that for types (b) and (c), we have $G/N\cong \D_6$, again a contradiction with (P1).

Suppose that $G$ satisfies the conditions of \cite[Theorem~4.7]{baniq}.
Then $G=R\times G_{m,r}$ with $\gcd(mn,30)=1$ and $R$ is a subgroup of $\B^*$ containing $O_2(\B^*)$ where $\B^*$ is the following extension of $\B$:
    $$\B^*=\B\rtimes C_2 = ((\GEN{i,j}_{\q_8}\Ydown_2 \GEN{a,b}_{\D_8})\nsp_2 \GEN{u,v}_{\SL(2,5)})\rtimes \GEN{h}_2$$
where $i,j,a,b,u$ and $v$ satisfy the relations of $\B$ and the action of $h$ on $\B$ is given by
    $$j^h=ij,i^h=i\inv,a^h=a,b^h=b\inv,u^h=u^3,v^h=u^{vu}v^2.$$
By (P4) we deduce that $G_{m,r}=1$ and hence $O_2(\B^*)\subseteq G \subseteq \B^*$.
Moreover, \cite[Theorem~4.7]{baniq} also states that either $G \subseteq \B$ and $G/O_2(G)$ is isomorphic to $\D_6, C_5$ or $\D_{10}$ or $G\not\subseteq \B$ and $G/O_2(G)\cong C_5\rtimes C_4$. The  case $G/O_2(G)\cong \D_6$ contradicts (P1)  and inspecting the proof \cite[Theorem~4.7]{baniq} one observes that in the last case the only two-by-two matrix algebra containing $G$ is $M_2(\HQ(\Q(\sqrt{2})))$, contradicting (P2). Therefore we have $O_2(\B^*)\subseteq G\subseteq \B$ and $G/O_2(G)$ is isomorphic to either $C_5$ or $\D_{10}$.
Moreover, $O_2(\B^*)=O_2(\B)=\GEN{i,j,a,b}$, $\B/\O_2(\B)\cong A_5$ and $\B^*/O_2(\B^*)\cong S_5$.
Furthermore all the subgroups $H$ of $S_5$ with $H/O_2(H)\cong C_5$ (respectively, $H/O_2(H)\cong \D_{10}$) are conjugate in $A_5$ to $\GEN{(1,2,3,4,5)}$ (respectively, $\GEN{(1,2,3,4,5),(2,5)(3,4)}$).
This implies that $O_2(G)=O_2(\B)$ and there are exactly two conjugacy classes of subgroups $G$ of $\B^*$ satisfying the conditions $O_2(\B)\le G$ and $G/O_2(\B)\cong C_5$ or $\D_{10}$.
A computer calculation using GAP shows that if $G/O_2(G)\cong C_5$ then $G$ is  the group of item (\ref{G160}) in Theorem~\ref{Main} and if $G/O_2(G)\cong \D_{10}$ then $G$ is the group of item (\ref{G320}) in Theorem~\ref{Main}. This finishes this case.

Finally assume that $G$ satisfies one of the conditions (a)-(i) in \cite[Theorem~5.8]{baniq}.
In case (e), $G$ contains an element of order 20 and, in case (i), $G$ contains an element of order 24.
In case (d), $G$ contains $\SL(2,5)\Ydown_2 C_m$ with $4\mid m$ and thus $G$ has an element of order 20. These three cases are hence excluded by (P4).
In case (a), $G=\SL(2,5)\Ydown_2 \D_{2^{\alpha}m}$ with either $m=1$ and $\alpha\ge 4$ or $m>1$ odd and $\alpha\ge 2$.
However, by (P4), $m=1,3$ or $5$ and as $\SL(2,5)$ has elements of order 3 and 5, if $m\ne 1$ then $G$ has an element of order 15, in contradiction with (P4). Therefore $m=1$ and, again using (P4) we deduce that $\alpha= 4$. Then $G/\SL(2,5)\cong \D_8$, in contradiction with (P1).
In case (c), $G=\A^{\pm}\times G_{m,r}$.
By (P1), $G=\A^{\pm}$ and hence $G$ is as in item (\ref{G240-1}) or item (\ref{G240-2}) of Theorem~\ref{Main}.
The same argument shows that in case (f), $G=\SL(2,9)$ and in case (h), $G=\B$.
So in these cases $G$ either is as in item (\ref{SL9}) or as in item (\ref{G1920}) of Theorem~\ref{Main}.
Assume that $G$ satisfies condition (b).
Then, $G=\SL(2,5)\Ydown_2 G_{m,r}$ with $2\mid s$.
Assume that $M_2(D)$ is an exceptional component of $\Q G$. As $G$ is an epimorphic image of the direct product $\SL(2,5)\times G_{m,r}$, there are simple components $A$ of $\Q \SL(2,5)$ and $B$ of $\Q G_{m,r}$ such that $M_2(D)$ is an epimorphic image of $A\otimes_{\Q} B$ and $\SL(2,5)$ is contained in $A$ and $G_{m,r}$ is contained in $B$.
A dimension argument compared with the Wedderburn decomposition of $\Q \SL(2,5)$ obtained in the proof of Proposition~\ref{Suficiencia},  shows that $A$ is either $M_2\quat{-1,-3}{\Q}$ or $(\Q(\zeta_5)/\Q(\sqrt{5}),-1)$. However, in the second case the centre of $D$ contains $\Q(\sqrt{5})$ in contradiction with the fact that $M_2(D)$ is an exceptional component. Thus $A=M_2\quat{-1,-3}{\Q}$ and hence $B=\Q$.
 Therefore $G_{m,r}$ is contained in $\Q$ and hence it is $C_2$. Thus $G=\SL(2,5)$ in contradiction with the assumption that $G$ is not contained in a division algebra.
Suppose finally that $G$ is of type (g). Then $G$ contains a subgroup $N$ of index 2 isomorphic to $\SL(2,9)$ and an element $g\in G\setminus N$ such that $g^2\in Z(N)$ and $g$ acts on the entries of the elements of $\SL(2,9)$ as the Frobenius automorphism $x\mapsto x^3$. This is the group identified with [1440,4591] in the GAP library of small group. Using Wedderga we obtain the Wedderburn decomposition
    \begin{eqnarray*}
    \Q G &=& 2\Q\oplus 4 M_5(\Q) \oplus M_3(\Q(\zeta_3)) \oplus M_2(\HQ(\Q(\sqrt{3}))) \oplus 2 M_9(\Q) \oplus \\
            && 2 M_{10}(\Q) \oplus M_4(\Q(\zeta_5)/\Q,-1) \oplus M_{16}(\Q) \oplus M_{10}(\HQ(\Q)).
    \end{eqnarray*}
which has not any exceptional component, in contradiction with the hypothesis. This finishes the proof of the Proposition and of Theorem~\ref{Main}.
\end{proof}

\section{Apendix: GAP programs used in the proof of Proposition~\ref{CSPMetabelian}}\label{Program}

In this appendix we include the GAP code used in the proof of Proposition~\ref{CSPMetabelian}.
The following GAP program implements the function \verb+Propiedad+ with three arguments $G, H$ and $K$.
If $(H,K)$ is a strong Shoda pair of a finite group $G$ then \verb+Propiedad(G,H,K)+ returns true if one of the conditions (1)-(7) of Proposition~\ref{SSPComponent} holds.
The groups are  identified using the terminology of the GAP library of small groups.

\bigskip

\begin{verbatim}
Propiedad := function(G,H,K)

local id,N;

if G=H then
  return false;
fi;
id := IdSmallGroup(G);
N:=Normalizer(G,K);
if Size(K)=1 and id in [[6,1] , [8,3] , [12,4]] then
  return true;
fi;
if Size(K)=2 and id = [8,3]  then
  return true;
fi;
if Size(K)=1 and id in [[16,6] , [24,1] , [24,5]] then
  return true;
fi;
if Size(H)=4*Size(K) and Size(G)=2*Size(N) and N=H and
            (Size(K)=2 or Size(K)=4) then
  return true;
fi;
if Size(K)=1 and id = [16,8] then
  return true;
fi;
if Size(K)=1 and id in [[24,10] , [24,11] ] then
  return true;
fi;
if (Size(H) = 3*Size(K) or Size(H)=6*Size(K)) and
        Size(G)=2*Size(H) and N=H and Size(K) <> 1 and
        Size(H) mod Size(K)^2 = 0 then
  return true;
fi;
if Size(K)=1 and id = [40,3] then
  return true;
fi;
if Size(K)=1 and id = [32,42] then
  return true;
fi;
if Size(H) = 4*Size(K) and Size(G)=2*Size(N) and
        IdSmallGroup(N/K) = [8,4] and (Size(K)=2 or Size(K)=4) then
  return true;
fi;
if 6*Size(K) mod Size(H) = 0 and Size(G)=2*Size(N) and
        IdSmallGroup(N/K) = [12,1] and Size(K) <> 1 and
        Size(H) mod Size(K)^2 = 0 then
  return true;
fi;
return false;
end;
\end{verbatim}

The program below computes the list \verb+L+ formed the groups $G$ of orders 6, 8, 12, 16, 18, 24, 32, 36, 40, 48, 64, 72 or 144 for which \verb+Propiedad(G,H,K)+ returns true for some strong Shoda pair $(H,K)$ of $G$ and another list \verb+R+ with the groups in \verb+L+ with no proper quotient isomorphic to any of the groups in items (\ref{D6})-(\ref{G72}) of Theorem~\ref{Main}. The list \verb+L+ contains 121 groups while the resulting list \verb+R+ contains the 8 groups in items (\ref{D6})-(\ref{G72}) of Theorem~\ref{Main}.

\bigskip

\begin{verbatim}
LoadPackage("wedderga");
D := [6,8,12,16,18,24,32,36,40,48,64,72,144];
csp := [[6,1],[8,3],[16,6],[16,13],[24,11],[32,50],[40,3],[72,19]];
L := [];
R := [];

for n in D do
 m:=NumberSmallGroups(n);
 Print("\n",[n,m]);
 for i in [1..m] do
  G:=SmallGroup(n,i);
  if not IsAbelian(G) then
   Print("\n",i);
   prop := false;
   ssp := StrongShodaPairs(G);
   nssp := Length(ssp);
   j:=0;
   while not prop and j<nssp do
    j:=j+1;
    x:=ssp[j];
    H:=x[1];
    K:=x[2];
    prop := Propiedad(G,H,K);
   od;
   if prop then
    Add(L,[n,i]);
    NS := Filtered(NormalSubgroups(G),x->Size(x)>1);
    IdNS := SSortedList(NS,x->IdSmallGroup(G/x));
    if Intersection(IdNS,csp)=[] then
     Add(R,[n,i]);
    fi;
   fi;
  fi;
 od;
od;
\end{verbatim}

\bibliographystyle{amsalpha}
\bibliography{ReferencesMSC}

\end{document}